\documentclass[11pt]{article}

\usepackage{latexsym,mathrsfs}
\usepackage{amsmath,amssymb} 
\usepackage{amsthm,enumerate,verbatim}
\usepackage{amsfonts}
\usepackage{graphicx}
\usepackage{algorithm}
\usepackage{booktabs}
\usepackage{enumitem}
\usepackage[noend]{algpseudocode}
\usepackage{fullpage}

\usepackage{hyperref} 

\usepackage{url}
\usepackage{color}
\usepackage{xcolor}

\newtheorem{lemma}{Lemma}

\newtheorem{theorem}{Theorem}

\newtheorem{definition}{Definition}
\newtheorem*{definition*}{Definition}
\newtheorem{example}{Example}

\newtheorem{remark}{Remark}

\DeclareMathOperator{\argmax}{argmax}

\DeclareMathOperator{\rank}{rank}

\DeclareMathOperator{\vect}{vec} 

\newcommand{{{{\trasp}}}}{^*}
\newcommand{\R}{\mathbb{R}}

\newcommand{\cA}{\mathcal{A}}

\definecolor{brightpink}{rgb}{1.0, 0.0, 0.5}
\newcommand{\ngc}[1]{{\color{brightpink} (\textbf{NG:} #1)}}

\definecolor{blue}{RGB}{0,0,255}

\title{Computing the Nonnegative Low-Rank Leading Eigenmatrix and \\ its Applications to Markov Grids and Metzler Operators} 
 
\date{}

\author{
 Nicolas Gillis\thanks{
 University of Mons, Rue de Houdain 9, 7000 Mons. Email: nicolas.gillis@umons.ac.be. NG acknowledges the support  by the European Union (ERC consolidator, eLinoR, no 101085607).} 
\qquad
 Carmela Scalone\thanks{University of L'Aquila, Italy. 
E-mail: carmela.scalone@univaq.it.} 
	}

\begin{document}

\maketitle

\begin{abstract}
We consider in this paper the problem of computing a nonnegative low-rank approximation of the rightmost eigenpair of a linear matrix-valued real operator. We propose an algorithm based on the time integration of a suitable differential system, whose solution is parametrized according to a nonnegative factorization. 
The conservation of the nonnegativity is theoretically motivated by the Perron-Frobenius theorem, while the computation of the rightmost eigenpair is motivated by two applications: 
(1)~a new class of Markov chains, which we called Markov grids, whose transition matrices can be decomposed as the sum of Kronecker products, and 
(2)~spatially structured systems in growth–diffusion operators arising for example in population and epidemic dynamics. 
Theoretical analysis and computational experiments show the effectiveness of the algorithm compared to standard approaches. 
\end{abstract}

%\textbf{Keywords.} 

\section{Introduction} 

%In this paper, we are interested in the following problem: 
Given a linear matrix-valued real operator, $\mathcal A : \R^{n \times n} \mapsto \R^{n \times n}$, the problem of finding its rightmost eigenvalue requires to find $\lambda$ with the largest real part such that $\mathcal A (X) = \lambda X$, for some non-zero matrix $X \in \R^{n \times n}$ called the leading eigenmatrix.  
The leading eigenmatrix  can often be assumed to be of low rank, or to be well approximated by a low-rank matrix; see the survey~\cite{simoncini2016computational} and the references therein; see also Theorem~\ref{th:rankonesol} for a particular case. 
This fact can be leveraged to design significantly more efficient algorithms, especially when we are facing large-scale problems, by reducing the computational cost and the memory requirements; see~\cite{kressner2014low, kressner2016low, simoncini2016computational, rakhuba2018jacobi, guglielmi2021computing}.  
In particular, in \cite{guglielmi2021computing}, the authors introduce a suitable ordinary differential equation (ODE), whose solution allows one to approximate the leading eigenmatrix of the linear operator.    
For a general linear operator, under generic assumptions, the solution of the considered matrix differential equation converges globally to its leading eigenmatrix. 
The search of low-rank solutions relies on the dynamical low-rank approach \cite{koch2007dynamical}, where the matrix ODE  is projected onto the tangent space of the manifold of  matrices of a prescribed rank. 
The numerical solution is obtained by using the dynamical low-rank based projector splitting integrator~\cite{lubich2014projector}. The link between eigenvalue problems and dynamical systems is very deep in the literature, see, e.g., \cite{absil2006continuous, chu1988continuous, embree2009dynamical, nanda1985differential}. 

In this paper, we are interested in nonnegative linear operators, that is, $\mathcal A (X) \geq 0$ for any $X \geq 0$, where $X \geq 0$ means that $X$ is component-wise nonnegative. One application that motivates our work is a special type of Markov chains that we introduce in this paper and that we will refer to as Markov grids; see Section~\ref{sec:grids}. The leading eigenmatrix is always nonnegative in this case and the leading eigenvector is real, simple, and positive, by the Perron–Frobenius theorem. 
More generally, one may consider Metzler operators, which can naturally arise from semi-discretizations of positive partial differential equation (PDE)  dynamics. In this setting, the positivity of the dominant eigenvector is guaranteed by the Perron--Frobenius theorem, providing a rigorous justification for applying our low-rank, positivity-preserving, algorithm beyond nonnegative matrices; see Section~\ref{sec:metzler}. 

The best low-rank approximation of a nonnegative matrix is not necessarily nonnegative; see, e.g., Example~\ref{ex:neglowrank33}. 
This very fact has inspired 
    a recent line of works by Song, Ng et al.~\cite{song2020nonnegative, song2022tangent} who designed algorithms to find the best nonnegative low-rank approximation of a nonnegative matrix; more precisely, they aim at solving the following problem: Given a matrix $M \geq 0$ and a prescribed rank $r$, solve 
    \[
\min_{X} \|M - X\|_F^2   
\quad 
  \text{ such that } 
  \quad
  X \geq 0 \text{ and } \rank(X) \leq r. 
    \] 
%for a matrix $M \geq 0$ and a  rank $r$, 
In some applications, such as Markov chains for which the entries of $X$ can be interpreted as probabilities, it is crucial to keep nonnegativity of the eigenmatrix. 
In this paper, we therefore design an algorithm to find a low-rank approximation to the rightmost eigenvalue problem of a nonnegative linear operator  with nonnegativity constraints. 
Given a linear operator $\mathcal A$ and  rank $r$, this problem can be formulated as the computation of the 
 rightmost eigenpair \((\lambda, X)\) such that 
 %with \(\lambda \in \mathbb{R}\), \(X \in \mathbb{R}^{n \times n}\) solving 
\begin{equation}
    \label{eq:mainprob}
   % \text{computing the rightmost} \((\lambda, X)\) 
%\max_{\lambda, X} \lambda 
\mathcal{A}(X) = \lambda X, \;
\operatorname{rank}(X) \leq r, \;
X \geq 0,\;
X \neq 0. 
% \begin{cases}
% \mathca{A}(X) = \lambda X \\
% \lambda \text{ is maximum} \\
% \operatorname{rank}(X) \leq r \\
% X \geq 0 \\
% X \neq 0
% \end{cases}
\end{equation} 
% \begin{equation} \label{eq:mainprob}
%   \text{Compute}\quad (\lambda,X),
%   \quad 
%   \text{ such that } 
%   \quad
% \mathcal A (X) \approx \lambda X, \rank(X) \leq r \text{ and } X \geq 0.   
% \end{equation} 
The starting point of our algorithm is the differential system considered in \cite{guglielmi2021computing}. However, the adaptation that combines low-rank and nonnegativity constraints  is nontrivial. The reason is that a dynamical low-rank approximation is used to include the low-rank constraint on the solution, see \cite{koch2007dynamical}. Such approach is based on a decomposition of the solution, which involves orthogonality conditions that do not combine well with nonnegativity; see \cite{koch2007dynamical, guglielmi2021computing, guglielmi2020efficient}. In fact, a matrix is nonnegative and its columns are orthogonal if and only if it has a single non-zero entry in each row. 

To handle this situation, we parametrize the solution using  nonnegative matrix factorization (NMF)~\cite{gillis2020nonnegative}. This is in contrast to the rationale used in~\cite{song2020nonnegative, song2022tangent} where authors rely on alternating projections on the space of nonnegative matrices and on the manifold of low-rank matrices. 
The latter is accomplished with the truncated singular value decomposition (SVD), which would increase the cost of the time integration of the ODE, if we wanted to use a similar technique.

\paragraph{Outline and contribution of the paper} 

Section~\ref{sec:motiv} motivates the study of the form~\eqref{eq:mainprob} with two case studies. First, we introduce a new class of Markov chains, which we refer to as Markov grids; see Section~\ref{sec:grids}.   
In a nutshell, Markov grids are Markov chains where the transition matrix, $P \geq 0$, can be decomposed as $P = \sum_{p=1}^t \alpha_p B_p \otimes A_p$, where $\otimes$ is the Kronecker product, $\alpha_p$ belongs to the probability simplex, and $(A_p, B_p)$ are transition matrices (that is, row stochastic matrices).
Then, we explain that, more generally, the problem~\eqref{eq:mainprob} can be used for a wider class of linear operators that satisfy a Perron-Frobenius property, including Metzler operators arising, for instance, from spatially discretized reaction-diffusion PDEs; see Section~\ref{sec:metzler}.  
This broadens the potential applications beyond Markov chains, covering problems in population dynamics, epidemiology, and other spatially structured systems where positivity of the principal eigenmatrix is physically meaningful. 
Section~\ref{sec:algo} reviews an ODE-based algorithm for solving~\eqref{eq:mainprob} for any linear operator $\cA$, which we then extend to the nonnegative case. 
%The method proposed in this paper is suitable for any linear operator $\cA$, but one of the structures that is most interesting to us is the one present in , which is very close to the special Markov chains we consider.
%Section~\ref{sec:grids} presents a special type of Markov chains, which we refer to as Markov grid, that provide a motivation for solving~\eqref{eq:mainprob}. 
Section~\ref{sec:numexp} provides numerical experiments illustrating the effectiveness of our algorithm.

\section{Motivations: Markov grids and Meltzer operators} \label{sec:motiv} 

Let us motivate the study of~\eqref{eq:mainprob} with two practical examples: Markov grids and Meltzer operators. 

\subsection{Markov grids}
\label{sec:grids} 

In this section, we introduce a new notion, that of Markov grid, which motivates the study of finding a nonnegative low-rank leading eigenmatrix, that is, of solving~\eqref{eq:mainprob}. 
Let us start with a simple example, before providing a general definition.  
\begin{example}[Illustrative example]  \label{ex:33grid}
    Let us consider a simple 3-by-3 grid. 
    Each node represent a state in a discrete Markov chain. 
    %\footnote{\url{https://www.chegg.com/homework-help/questions-and-answers/application-markov-chains-model-movements-around-grid-con-sider-following-transition-diagr-q80269884}}  
    For each node, there is a 50\% probability to move to either a node in the same row or in the same column. It can only move to neighboring nodes, and, in the same row/column, they have the same probability to be picked; see Figure~\ref{fig:grid33}. 
    \begin{figure}[ht!]
	\begin{center}
		\includegraphics[width=0.25\textwidth]{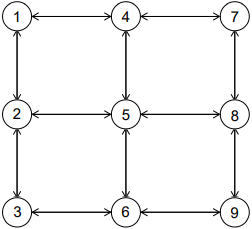}   
		\caption{Markov chain on a 3-by-3 grid.\label{fig:grid33}
   }  
	\end{center}
\end{figure}
 For example, from node~2, we have 50\% chance to move to node~5, and 25\% to move 
 to node~1 or~3. 
 
 Let $x_k \in \mathbb{R}^9_+$ be the state vector such that $x_k(i)$ is the probability to be in state $i$ after $k$ steps. We have $e^\top x_k = 1$ for all $k \in \mathbb{Z}_+$, where $e$ is the vector of all ones, and $x_0$ is the initial state vector.  
The transition matrix $P \in \mathbb{R}^{9 \times 9}$ is defined such that $P(i,j)$ is the probability to move from node~$i$ to node~$j$. 
%Note that we use an unconventional notation, as usually $P(i,j)$ is the probability to move from node $i$ to node $j$. 
The transition matrix $P \in \mathbb{R}^{9 \times 9}$ is row stochastic, that is, $P \geq 0$ and $P e = e$. 
We have $x_{k+1} = P^\top x_k$ for all $k \in \mathbb{Z}_+$, with 
\[ 
P 
%=  \frac{1}{2} I_3^\top \otimes A + \frac{1}{2} B \otimes I_3  
= 
\frac{1}{4} \left( \begin{array}{ccccccccc} 
  0 &   2 &   0 &   2 &   0 &   0 &   0 &   0 &   0 \\ 
  1 &   0 &   1 &   0 &   2 &   0 &   0 &   0 &   0 \\ 
  0 &   2 &   0 &   0 &   0 &   2 &   0 &   0 &   0 \\ 
  1 &   0 &   0 &   0 &   2 &   0 &   1 &   0 &   0 \\ 
  0 &   1 &   0 &   1 &   0 &   1 &   0 &   1 &   0 \\ 
  0 &   0 &   1 &   0 &   2 &   0 &   0 &   0 &   1 \\ 
  0 &   0 &   0 &   2 &   0 &   0 &   0 &   2 &   0 \\ 
  0 &   0 &   0 &   0 &   2 &   0 &   1 &   0 &   1 \\ 
  0 &   0 &   0 &   0 &   0 &   2 &   0 &   2 &   0 \\ 
\end{array} \right). 
\]

Because of the particular form of this chain, we could instead store the state vectors in 3-by-3 matrices, with $X_k \in \mathbb{R}^{3 \times 3}$, such that $x_k = \vect(X_k)$ for all $k \in \mathbb{Z}_+$, where $\vect(\cdot)$ vectorizes a matrix column-wise, that is, for all $\ell, i, j$,  
 \[
 x_{k}(\ell) = X_{k}(i,j) \quad \text{ for } \quad \ell = i + 3(j-1).  
 \] 
The iterative process in matrix form is given by 
\[ 
X_{k+1} = \mathcal{A} (X_{k}) = \frac{1}{2} A^\top X_{k} + \frac{1}{2} X_{k} A, 
\]
where 
\[
A = \begin{pmatrix}
0 & 1 & 0 \\
0.5 & 0 & 0.5 \\
0 & 1 & 0      
\end{pmatrix}.  
\] 
The matrix $A$ describes the transition between rows/columns of the Markov grid; see  
Figure~\ref{fig:grid13}. 
    \begin{figure}[ht!]
	\begin{center}
		\includegraphics[width=0.25\textwidth]{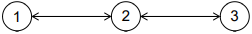}   
		\caption{Markov chain with 3 states. \label{fig:grid13}}  
	\end{center}
\end{figure} 
The matrix representation is more compact as it only requires to store the $n$-by-$n$ matrix $A$, as opposed to the $n^2$-by-$n^2$ matrix $P$ (not taking into account sparsity), while computing $\mathcal{A} (X_{k})$ requires $O(n^3)$ operations as opposed to $O(n^4)$ to compute $P^\top x_k$.  If the $X_k$'s have rank $r$, the cost reduces to $O(n^2 r)$ operations; see Section~\ref{sec:algo} for more details. 
\end{example}

Let us formally define the notion of Markov grids. 

\begin{definition}[Markov grid]
For $m, n \in \mathbb{N}$, a Markov grid with $mn$ states is a discrete Markov chain with the following properties: 
\begin{itemize}

\item The state vector at step $k$, $x_k \in \mathbb{R}^{mn}$, that is, the vector containing the probabilities to be in the different states at step $k$, can be represented by a matrix $X_k \in \mathbb{R}^{m \times n}$, with $x_k = \vect(X_k)$, $X_k \geq 0$ and $e^\top X_k e = 1$, for all $k \in \mathbb{Z}_+$.  

    \item The transition from $X_{k+1}$ to $X_k$ is defined as follows 
    \[
X_{k+1} = \mathcal A(X_k) = \sum_{p=1}^t \alpha_p A_p^\top X_{k} B_p, 
\] 
where 
\begin{itemize}
    %\item $X_{k} \in \mathbb{R}^{m \times n}$ contains the $mn$ states $X_k(i,j)$ in a matrix form,  and $\sum_{i,j} X_k(i,j) = 1$ for all $k$ or, equivalently, $e^\top X_k e = 1$ for all $k$ where $e$ is the vector of all ones. 

    \item $A_p \in \mathbb{R}^{m \times m}_+$ and $B_p \in \mathbb{R}^{n \times n}_+$ are transition matrices (row stochastic matrices).  

    \item $\alpha \in \mathbb{R}^t_+$ are the weights to balance the $t$ terms in the chain, with $e^\top \alpha = 1$. 
\end{itemize}
\end{itemize}
\end{definition}

The constraints on $A_p$, $B_p$ and $\alpha_p$ ensure that if $X_0 \geq 0$ satisfies $e^\top X_0 e = 1$, then $X_k \geq 0$ and $e^\top X_k e = 1$ for all $k \in  \mathbb{Z}_+$, since 
\[
e^\top \left( \sum_{p=1}^t \alpha_p A_p^\top X_{k+1} B_p \right) e 
= 
 \sum_{p=1}^t \alpha_p e^\top A_p^\top X_{k+1} B_p e 
 = 
 \sum_{p=1}^t \alpha_p e^\top X_{k+1} e = 1. 
\] 

Note that if $m$ or $n$ is equal to one, a Markov grid reduces to a standard Markov chain, and $t$ can be chosen equal to one w.l.o.g.

\paragraph{Interpretation} 

Let us try to interpret the transition $X_{k+1} = A_p^\top X_{k} B_p$. 
The product $A_p^\top X_{k}$ makes transitions along the columns of $X_{k}$, since $(A_p^\top X_{k})_{:,j} = A^\top X_{k}(:,j)$ for all $j$. 
The entries in each column of $X_k$ have the same probabilities of transitions towards states in the same column. 
It is like there are $n$ copies of the same chain with $m$ states, each corresponding to a column of $X_k$, and each performing transitions with the same probabilities in parallel.  
Similarly, the product $X_{k} B_p$ performs the transitions with the same probabilities along the rows of $X_k$, since $(X_{k} B_p)_{i,:} = X_{k}(i,:) B_p$ for all $i$. It is like there are $m$ 
copies of the same chain with $n$ states, one for each row of $X_k$, each performing the same transitions in parallel.  
The most natural example is, of course, a grid-like chain, as shown in Example~\ref{sec:grids}. Note that there could be transitions from non-adjacent states on the same row and column.

\paragraph{Link with standard Markov chains} 
We can vectorize the $X_{k}$'s to fall back on a standard Markov chain, using the identity $\vect(AXB) = \left(B^\top \otimes A\right) \vect(X)$, where $\otimes$ is the Kronecker product between two matrices, so that 
\begin{equation} \label{eq:mattovec}
    \vect(X_{k+1}) 
= 
\underbrace{\left( \sum_{p=1}^t \alpha_p B_p^\top \otimes A_p^\top \right)}_{P^\top} \vect(X_{k}), 
\end{equation}
where $P \in \mathbb{R}^{mn \times mn}$ 
is the transition matrix of a standard Markov chain. 
In Example~\ref{sec:grids}, one can check that 
$P = \frac{1}{2} I_3 \otimes A^\top + \frac{1}{2} A^\top \otimes I_3$, since, in this example, $A_1 = A$, $B_1 = I_3$, $A_2 = I_3$, $B_2 = A$, and $\alpha_1 = \alpha_2 = 1/2$.   

Hence a Markov grid is a Markov chain where the transition matrix can be written in the form~\eqref{eq:mattovec}. 
The same structure, but without row stochasticity, was used in~\cite{rakhuba2018jacobi} to solve eigenvalue problems where the eigenvector is reshaped as a low-rank matrix. 

\begin{remark}[Kronecker factorization] 
    The decomposition~\eqref{eq:mattovec} suggests the following problem: given a matrix $P$, 
    find an approximation written as the sum of $t$ Kronecker products:  
    \[
P \; \approx \; \sum_{p=1}^t \alpha_p B_p \otimes A_p. 
    \]
    % Decomposition of a matrix  
    In the unconstrained case and $t=1$, this problem was studied by Van Loan and Pitsianis~\cite{VanLoan1993}, and can be solved in closed form via an SVD computation. The case $t > 1$ has been studied extensively and can be solved, for example, using an iterative procedure optimizing each term $\alpha_p B_p \otimes A_p$ sequentially; see, e.g., \cite{Hackbusch2005, cai2022kopa, hameed2022convolutional, ChenTanh2024} and the references therein. 
    
    However, as far as we know, the case with nonnegativity constraints has not been addressed in the literature, this is a topic of further research.     
    %is also related to tensor decompositions where unfoldings have the form of Kronecker products
    As the derivations above show, in the case of nonnegative matrices, this problem is equivalent to finding the nearest Markov grid to a given Markov chain. 
    Moreover, it could be interesting to consider the decomposition of $P$ as the sum of Kronecker products of different sizes, meaning that the underlying Markov chain is decomposed into Markov grids with different structures/connections.   
\end{remark}

\paragraph{Computational advantages}   

When $t$ is small, it is more convenient to keep the matrix format; constructing the matrix $P$ explicitly would be prohibitively expensive, with $O\big( (mn)^2 \big)$ memory and $O\big( (mn)^4 \big)$ operations when performing matrix-vector products. 
 Of course, in many Markov chains, $P$ is sparse and hence the computational load is significantly lower.  
 However, computing a low-rank approximation to the system $X = \mathcal A(X)$ would probably be less straightforward in the vector representation. 
Low-rankness can ``filter part of the noise'' in the chain, because we assume the underlying solution should have a low rank.

%\ngc{I am not sure what you mean by shift, but the eigenvalue 1 is the rightmost eigenvalue so whatever worked before should work for this particular case.} \textcolor{blue}{The main tool of \cite{guglielmi2021computing} is the dynamical system, which is the "continuous analogous" of the power method. Since $\lambda =1 $ is the rightmost we are fine with the general method. I run some examples with the code, directly imposing in all the iterations that the eigenvalues is 1.It seems to me that it takes many many fewer iterations.}\\ This is related to matrix equations~\cite{simoncini2016computational}. 

%Has this been studied before? Has algorithms to compute low-rank solutions been developed?  \textcolor{blue}{ I checked but I didn't find anything low rank solutions for this particular case. Perhaps the closest reference is this one \cite{kressner2014}, where they clearly pose the problem of combining low rank and non negativity.}\\

%Motivations? Such Markov chains would make sense if the underlying graph is close/similar to a grid where nodes are connected in the rows and columns of $X$. However, this requires to have groups of $mn$ nodes connected in a rather specific way, with $m$ groups of $n$ nodes and $m$ groups of $n$ nodes. However, all the nodes in each group do not need to interact. Interesting question: Given a large Markov chain, can we detect such connections? Probably this is hard, would make more sense to construct it. 

%It would be nice to find some illustrative examples. 

\paragraph{Example: Tracking a moving object}

Assume an object is moving randomly in a given space, e.g., a person randomly walking in a city. 
We can subdivide the space into a grid, and model that movement as a Markov chain on the grid. Of course, this requires to keep the symmetry in the grid, that is, the probabilities to move along states on the same row (resp.\ column) should be the same regardless of the column (resp.\ row) we are in.    

%Moreover, using terms of the form $A X_k B$, we can add transitions from nodes to nodes, e.g., if there is a subway system, we can model the person taking the subway with some probability, and moving to another subway station with some probability. 
% In our 3-by-3 grid, assume there is a possible to move from (1,1) to (3,3) in one step, and that the probability to pick that transition is 1/3. We will add a term in the model of the form $A_s X_k B_s$, where 
% \[
% A_s = \begin{pmatrix}
% 0 & 0 & 0 \\
% 0 & 0 & 0 \\
% 0 & 0 & 0      
% \end{pmatrix}, 
% \quad
% B_s = \begin{pmatrix}
% 0 & 0 & 0 \\
% 0 & 0 & 0 \\
% 0 & 0 & 0        
% \end{pmatrix} = A^\top.  
% \]
%\url{https://www.researchgate.net/publication/301312745_Influential_spatial_facility_prediction_over_large_scale_cyber-physical_vehicles_in_smart_city} 

For an object randomly moving in a 3D volume, we could generalize our constructions by representing the states as a tensor $\mathcal{X} \in \mathbb{R}^{n_1 \times n_2 \times n_3}$, and the transitions would be matrices applied on each dimension (a.k.a.\ mode) of the tensor. This is a topic for further research. Note that low-rank tensor representations have been used to approximate continuous Markov chains with a finite number of states in the context of tumor progression modeling; see~\cite{georg2023low, klever2025}.

\paragraph{Adaptation of the model} 

Of course, the Markov grid model is restrictive, since it requires a very structured chain. However, one could consider a Markov chain of the form 
    \[
P \; = \; \sum_{p=1}^t \alpha_p B_p \otimes A_p + S  \geq 0 , 
    \]
where $S$ is a sparse matrix with $Se=0$. This model is reminiscent to that of sparse + low-rank matrix decompositions~\cite{chandrasekaran2011rank, candes2011robust}, where the low-rank part is replaced by a Kronecker factorization. 

For example, consider the 3-by-3 grid from  Example~\ref{ex:33grid}. We could add a transition between states 1 and 9, e.g., imagine this 3-by-3 grid represents a city where a person performs a random walk, and there is a subway connection between nodes 1 and 9 that one chooses with probability 1/2. Then the new transition matrix, $P'$, can be obtained by  replacing the first and ninth row of $P$ as follows  
\[
P'(1,:) 
= \frac{1}{4} (0,1,0,1,0,0,0,0,2) 
= \frac{1}{4} P(1,:) + \frac{1}{4}   (0,-1,0,-1,0,0,0,0,2), 
\]
\[
P'(9,:) = 
 \frac{1}{4} (2,0,0,0,0, 1,0, 1,0) 
= \frac{1}{4} P(9,:)    + \frac{1}{4} (2,0,0,0,0,-1,0,-1,0).  
\] 
Again, this is out of the scope of this paper to compute such decompositions; this is a topic of further research.

% would happen 
%\url{https://www.nist.gov/system/files/documents/itl/antd/InterAgencyReport7566.pdf}

\paragraph{Low-rankness of solutions} 

Interestingly, for special classes of Markov grids, the stationary state has rank one.  
%There are the grids with the following form 
%For $\|X_0\|_1 = 1$, we have $\|X_k\|_1 = 1$ for all $i$. 
%This is a Markov chain where the states are the entries of $X$, while the transitions only allow to move along the same column ($A X_k$), the same row ($X_k B$), or both consequently ($A X_k B$), but the probably of moves is independent on the columns/rows we are at.  
\begin{theorem} \label{th:rankonesol}
Let $A$ and $B$ be irreducible transition matrices, 
and $\alpha \in \mathbb{R}^3_+$ be on the unit simplex with $\alpha_1 \neq 1$ and $\alpha_2 \neq 1$. 
Consider the Markov grid with transition 
\[
\mathcal A(X) = \alpha_1 A^\top X + \alpha_2 X B + \alpha_3  A^\top X B. 
\] 
Then the stationary state of this Markov grid , $X^*$, has rank-one, and is given by $X^* = \mu_A \mu_B^\top > 0$, 
where $\mu_A > 0$ and $\mu_B > 0$ are the stationary states of $A$ and $B$, respectively.  
    %Let the above Markov process be irreducible, that is, any state is reachable from any state with the same number of steps. 
    %Then the above iterative process converges to the stationary state, $X^* > 0$, such that $\rank(X^*) = 1$. In fact, $X^* = \mu_A \mu_B^\top$ where $\mu_A$ and $\mu_B$ are the stationary states of $A$ and $B$. 
\end{theorem} 
\begin{proof}
First, note that $\alpha_1 \neq 1$ and $\alpha_2 \neq 1$ implies that the Markov grid is irreducible, that is, one can transition from any node to any node. 
In fact, $A^\top X$ allows one to move anywhere along columns in a finite number of steps, while $X B$ allows one to move anywhere along rows in a finite number of steps, since $A$ and $B$ are irreducible. 

Then it is easy to verify that $X^* = \mu_A \mu_B^\top > 0$ is the (unique) stationary state: 
\[
\mathcal A(X^*) =  
\alpha_1 A^\top (\mu_A \mu_B^\top) + \alpha_2 (\mu_A \mu_B^\top) B + \alpha_3 A^\top (\mu_A \mu_B^\top) B 
= \mu_A \mu_B^\top,  
\]
since $A^\top \mu_A = \mu_A$ and $B^\top \mu_B  = \mu_B$.  
\end{proof}

\paragraph{Nonnegativity of low-rank solutions} 
As explained in the introduction, our goal would be to compute a low-rank approximation to the rightmost eigenpair $(\lambda, X)$ of the operator $\cA$.
%$\mathcal A(X) = \lambda X$ for the rightmost eigenvalue of $\mathcal A$. %Note that, for Markov grids, we will have $\lambda = 1$. 
For applications such as Markov grids, nonnegativity of the solution is crucial to be able to meaningfully interpret it. 
Note that stationary states of nonnegative linear operators are nonnegative. However, a nonnegative matrix may not have a nonnegative best rank-$r$ approximations, e.g., when it is sparse (or has sufficiently many small entries), or has some 
structure~\cite{song2020nonnegative, song2022tangent}. 
Hence it is better to enforce nonnegativity of the solution within the algorithm, as opposed to a suboptimal two-step approach that would project an unconstrained solution onto the set of low-rank nonnegative matrices. This will be illustrated in Section~\ref{sec:numexp}.

\begin{example} \label{ex:neglowrank33}
% Given an integer $n$ and $\epsilon < 1/n$, let us define the $n$-by-$n$ matrix 
% \[
% A_{\epsilon} = \left( \begin{array}{cccc}
%     1-(n-1) \epsilon  & \epsilon & \dots  & \epsilon  \\
%     1 & 0 & \dots & 0 \\  
%     \vdots & \vdots &  & \vdots \\  
%     1 & 0 & \dots & 0  
% \end{array} \right). 
% \]
% The stationary state of $A_{\epsilon}$ is $\mu_A =  \frac{1}{1+(n-1)\epsilon} [1,\epsilon,\dots,\epsilon]$.  
% Let us define the Markov grid   
% \[
%  X_{k+1} = \frac{1}{n} \sum_{p=1}^n A_p^\top X_k B_p , 
% \] 
% where, for all $p \in [n]$, $A_p = B_p = A_{\epsilon_p}(\pi_p,\pi_p)$ with $\pi_p = [p,p+1,\dots, n, 1, 2,\dots, p-1]$ are permutations. Hence the stationary state of each pair $(A_p,B_p)$, a permutation of $\mu_A \mu_A^\top$, is close to a matrix with most of its energy is located on the $k$th diagonal entry. 
% Hence the stationary state of the system above should be close to a diagonal matrix, when the $\epsilon_p$'s are small. 

% \ngc{Can we prove this class of matrices has negative entries in best rank-$r$ approximations?} 
Consider the Markov grid defined by the linear operator %as above with     $\epsilon = [0.1, 0.15, 0.2]$, so that  
\[
\mathcal{A} (X) = \frac{1}{3} \left( A_1^\top X B_1 + A_2^\top X B_2 + A_3^\top X B_3 \right),  
\]
where  
\[
A_1 = B_1 =
\begin{pmatrix}
0.9 & 0.05 & 0.05 \\
1 & 0      & 0      \\
1 & 0      & 0      
\end{pmatrix}, 
A_2 = B_2 =
\begin{pmatrix}
0      & 1 & 0      \\
0.1 & 0.8 & 0.1 \\
0      & 1 & 0      
\end{pmatrix}, 
A_3 = B_3 = \begin{pmatrix}
0      & 0      & 1 \\
0      & 0      & 1 \\
0.075 & 0.075 & 0.85
\end{pmatrix}. 
\]
The rightmost eigenvalue is $\lambda = 1$, and the corresponding (full rank) eigenmatrix with Frobenius norm equal to 1 is 
\[
X^* = \begin{pmatrix}
0.5936 & 0.0277 & 0.0253 \\
0.0277 & 0.5585 & 0.0310 \\
0.0253 & 0.0310 & 0.5753
\end{pmatrix} . 
\]
The best rank-$2$ approximation of $X^*$ is given by 
\[
X_2^* = 
\begin{pmatrix}
0.5787 & 0.1024 & -0.0212 \\
0.1024 & 0.1848 &  0.2638 \\
-0.0212 & 0.2638 & 0.4303
\end{pmatrix}, 
\quad \text{ with } \|X^* - X_2^* \|_F = 0.5336. 
\]
When we apply the dynamical low-rank method~\cite{guglielmi2021computing} to this system 
with rank $2$, we %get $\lambda_2 \approx 2.9928 $ and the following 
obtain the following approximation of the eigenmatrix 
\[
X_2 = \begin{pmatrix}
0.6907 & 0.0998  & -0.0172 \\
0.0998 & 0.2563  &  0.3325 \\
-0.0172 & 0.3325 &  0.4644
\end{pmatrix}, \quad \text{ with } \|X^* - X_2 \|_F = 0.5558, 
\]
which is only slightly worse than the best rank-2 approximation of $X^*$, $X_2^*$. 
%The errors measured in the Frobenius norm are
% fpr which 
% \[ 
% \| X^*- X_2 \|_F = 0.5558  
% %\quad \text{and} \quad   \| X^*_2 - X_2 \|_F = 0.1683. 
% \]
% This is reasonable since the third singular value of $X^*$ is $\sigma_3 ^*  \approx 0.5337$. 
The method proposed in this paper, see Section~\ref{sec:algo}, 
%The "non negative analogous" %\cite{guglielmi2020efficient}+\cite{guglielmi2021computing} 
will give the following rank-2 nonnegative approximation of the  eigenmatrix 
\[
X_2^+ = \begin{pmatrix}
0.6967 & 0.0774 & 0 \\
0.0774 & 0.2693 & 0.3406 \\
0 & 0.3406 & 0.4450
\end{pmatrix}, \quad \text{ with } \|X^* - X_2 \|_F = 0.5561, 
\]
while  $\| X_2 - X_2^+ \|_F = 0.0481$. 
%Is $Y_2$ an approximate solution of the eigensystem. Approximation in what sense? Does it guarantee to find the best rank-$r$ solution of the true high-rank solution? Or something else? The dynamical system is (maybe???) trying to minimize $\| \mathcal{A} X - \lambda X \|_F^2$ ($\lambda = 1$ in this case)? \textcolor{blue}{I think I tried to put all in the section The method. If something is not clear, tell me.}
\end{example}

% \begin{remark}
% On larger examples, could nonnegativity be a measure of rank adaptation?\\
% On the one hand, we can force the solution to be low rank and not negative, as in \cite{guglielmi2020efficient}. On the other hand, we could see the presence of negativity as an index of too low rank and thus grow the rank directly during numerical integration. (perhaps)
% \end{remark}

%I agree that it is a bit hard to justify/motivate the development of a specific algorithm taking nonnegativity into account. However, it makes sense, and we can generate examples. Michael Ng has a line of work on this, finding the nearest nonnegative and low-rank matrix to a given nonnegative one; see, e.g., \cite{song2020nonnegative, song2022tangent} \textcolor{blue}{I will take a look to the papers.}. To be discussed and decided. However, what I find interesting and fun is the study of "Matrix Markov Chains". I have not seen this before. It could be a way to represent compactly certain types of Markov chains. 

%To compute the stationary state, we can use the power method, of course, which will converge linearly if the matrix $M$ is regular. Each iteration costs $O\big(t mn (m+n) \big)$ operations to compute the $t$ products $A_p X_{k+1} B_p$. If the $X_k = W_k H_k$'s are of low-rank $r$, we could reduce this cost significantly, for $r$ small, to $O\big(t r (mn + m^2 + n^2) \big)$. 

\subsection{Positive Dominant Modes in Spatially Structured Systems} \label{sec:metzler}

Reaction–diffusion equations provide a natural and significant framework for studying the spectral structure of growth–diffusion operators arising in population and epidemic dynamics~\cite{cantrell2003spatial}. In such systems, the long-term behavior of the population depends on the spectral properties of the underlying operator. In particular, the sign of the principal eigenvalue $\lambda_1$, that is, the eigenvalue with maximal real part, determines whether the population persists ($\lambda_1 > 0$) or becomes extinct ($\lambda_1 < 0$). This quantity acts as an effective spatial average of local demographic rates (birth, death, and dispersal), weighted by the spatial distribution of the population.

A classical  example is the KISS model (Kierstead–Slobodkin–Skellam), which describes diffusion and growth in a bounded habitat surrounded by a hostile exterior. In its simplest form,
\[
u_t = D \Delta u + r\,u \quad \text{in }\Omega, \qquad u|_{\partial\Omega} = 0,
\]
where $\Omega \subset \mathbb{R}^n$ is a bounded and sufficiently regular domain, $D > 0$ is the diffusion coefficient, and $r$ is the intrinsic growth rate. Seeking separable solutions of the form $u(x,t) = e^{\lambda t}\psi(x)$ leads to the spectral problem
\[
D\Delta \psi + r \psi = \lambda \psi, \qquad \psi|_{\partial\Omega} = 0.
\]
The principal eigenvalue $\lambda_1$ determines the asymptotic behavior of the system, while the associated eigenfunction $\psi_1$ describes the stationary spatial profile. According to the Krein-Rutman theorem, under standard ellipticity and regularity assumptions, $\lambda_1$ is real, simple, and its eigenfunction $\psi_1$ can be chosen strictly positive within $\Omega$.

Spatial heterogeneity can be incorporated by allowing space-dependent diffusion and growth coefficients: 
\[
u_t = \nabla \!\cdot ( D(x) \nabla u ) + r(x) u,
\]
where $D(x)$ is uniformly elliptic and $r(x)$ represents the local net growth rate. Different boundary conditions (Dirichlet, Neumann, or Robin) correspond to different ecological or physical settings and modify the spectral properties accordingly. Nevertheless, the qualitative behavior remains the same: the dominant eigenvalue governs persistence, and the corresponding eigenfunction retains strict positivity, see \cite{cantrell2003spatial}.

In spatially extended population or epidemic systems, analytical expressions for the dominant eigenpair are rarely available. Discretization of the operator, for instance through finite differences, leads to large sparse matrices whose dominant eigenpair approximates $(\lambda_1, \psi_1)$. 
From a physical or biological standpoint, the numerical scheme must preserve the positivity of the dominant eigenvector, since negative components lack meaningful interpretation in these contexts.

In many cases (see Section \ref{sec:numexp} for some examples), the discretized operator is a Metzler matrix, that is, a real matrix with nonnegative off-diagonal entries. For such matrices, the Perron–Frobenius theorem guarantees the existence of a simple real dominant eigenvalue with a strictly positive eigenvector. %Our low-rank positivity-preserving algorithm applies to other classes of matrices arising in these models. 

From a computational perspective, these considerations motivate the development of low-rank algorithms for approximating the dominant eigenpair while preserving the non-negativity of the eigenmatrix, ensuring intrinsic physical interpretability. 
Moreover, the structure of the operators arising from reaction–diffusion semi-discretizations — often of the form $A X + X A^\top + \dots$, with $A$ deriving from the Laplacian — naturally leads to solutions that can be well-approximated by low-rank matrices, see \cite{guglielmi2021computing}.

Finally, if the discretized operator is Metzler and Hurwitz (that is, with all eigenvalues have negative real parts), the implicit Euler scheme preserves the positivity of the flow of the linear semi-discrete system for $t>0$~\cite{hundsdorfer2003numerical}. 
From the numerical viewpoint, the spatial discretization of diffusion–reaction systems typically leads to stiff ODEs, especially on fine spatial grids. Implicit or semi-implicit time-stepping schemes are therefore recommended, as they maintain stability even for large diffusion coefficients. 
This correspondence naturally bridges continuous-space ecological or epidemic models with discrete-time dynamical systems, facilitating reliable numerical simulations.

In Section~\ref{sec:numexp}, we will provide a numerical example on a reaction diffusion model  with Neumann boundary conditions and spatially varying growth rate $r(x)$, and show how the low–rank positive algorithm proposed in Section~\ref{sec:algo} efficiently captures the principal eigenpair of the corresponding dicretized Metzler operator.

\section{Algorithm using differential equations} \label{sec:algo}  

%\ngc{What abou normalization of a solution, e.g., $e^\top X e = 1$? }
%\carmen{a normalization at the end should be fine, right?} \ngc{Yes, this is OK, but if the algorithm can take the constraint into account (probably hard?) it could be nice. If $X = WH$, $W \geq 0$, $H \geq $ and $e^\top X e = 1$. We need $e^\top W H e = 1$, which is implied by $e^\top W = e^\top /\sqrt{r}$ and $H e =e /\sqrt{r}$ which we can assume w.l.o.g.\ by the scaling degree of freedom. Is this doable in your algorithm?} \ngc{Actually it appears your normalize the solution anyway, so this is fine :-)}
%\carmen{Perfect! we normalize imposing norm 1, but the important thing is to maintain the solution with fixed norm. For instance, we can normalize as $X = X/ sum(X0)$, w.r.t. the initial condition, but this $X = X/ sum(X)$  changes the norm of the direction at any iteration and it is not very good.}

In this section, we propose an algorithm based on differential equations to solve the nonnegative low-rank leading eigenmatrix problem~\eqref{eq:mainprob}. 
Our algorithm exploits the analysis of the dynamical system studied in \cite{guglielmi2021computing}, which allows us to determine the sought eigenmatrix, as an asymptotically stable equilibrium point of the system. The low-rank approximation proposed in \cite{guglielmi2021computing} is not suitable for extension to the nonnegative case, therefore we propose a dynamical approach based on an NMF of the solution. 
To make the paper as self-contained as possible, we first recall, in Section~\ref{recap}, the results of \cite{guglielmi2021computing}, which will be the starting point for our proposed algorithm described in Section~\ref{sec:ouralgo}, and whose pseudocode along with a computational cost analysis is provided in Section~\ref{sec:pseudocode}. 

%In Section~\ref{}, we provide a peusdocode 

\subsection{Matrix ODE for the low-rank rightmost eigenpair}\label{recap}

%In this section, we briefly summarize the analysis presented in \cite{guglielmi2021computing}. 
Let $\mathcal{A}$ be a real matrix-valued linear operator,
%\begin{equation*}  \label{eq:linop}
$\mathcal{A} : \R^{n \times n} \mapsto \R^{n \times n}$. 
%\end{equation*}
%that is, if $X \in \R^{n \times n}$, then $ \mathcal{A}(X) \in \R^{n \times n}$.
%We are interested in computing its eigenpairs, that is
%We aim to find \textcolor{blue}{rightmost eigenpairs} of $\mathcal{A}$, i. e.  
A pair $(\lambda, X)$ is an eigenpair of $\mathcal{A}$ if it satisfies  
\begin{equation*} 
\mathcal{A}(X) = \lambda X, % \quad  \lambda \in \C, X \in \Cnn, 
\quad X \neq 0, % \mathcal{O} I do not like this notation, I just use 0 
\end{equation*} 
where $X$ and $\lambda$ are allowed to be complex-valued.
The aim is to find the rightmost eigenpairs of $\mathcal{A}$, assuming that $X$ has quickly decaying singular values. 
This assumption is supported by multiple examples of operators; see, e.g., \cite{guglielmi2021computing} and the references therein. 
This motivates to constrain the search of approximate eigensolutions on a low-rank manifold. 
Guglielmi et al.~\cite{guglielmi2021computing} considered the following system of ODEs  
\begin{equation} 
\label{Ode}
\left\{\begin{aligned}
\dot{X}(t) & = \mathcal{A}\left(X(t) \right) - \rho(X(t))\,X(t), 
\\
X(0) & =  X_0, \qquad \| X_0 \| = 1, 
\end{aligned} \right.
\end{equation}
for some real initial $X_0 \in \mathbb{R}^{n \times n}$, where  $\rho(X(t)) = \big\langle \mathcal{A}\left(X(t) \right), X(t) \big\rangle$. 
Is is shown that  the dynamical system \eqref{Ode} satisfies the following  properties :  
\begin{enumerate}[label=\textbf{P\arabic*.}, leftmargin=1.5em]

\item \textbf{Norm preservation.}  
The flow is confined to the unit norm sphere $\mathcal{B} := \{ Z \in \mathbb{R}^{n \times n} : \| Z \|_F = 1 \}$. That is, for any initial condition $X(0) \in \mathcal{B}$, the solution satisfies $X(t) \in \mathcal{B}$ for all $t \geq 0$.

\item \textbf{Characterization of equilibria.}  
A matrix $X \in \mathcal{B}$ is a fixed point of the flow if and only if it satisfies $\mathcal{A}(X) = \lambda X$, that is, $X$ is an eigenmatrix of $\mathcal{A}$.

\item \textbf{Stability of the dominant eigenmatrix.}  
Assume that $\mathcal{A}$ has a unique real eigenvalue $\lambda_1$ with maximal real part (rightmost). Let $V \in \mathcal{B}$ be an eigenmatrix associated with a simple real eigenvalue $\lambda$. Then $V$ is a locally asymptotically stable equilibrium of the flow if and only if $\lambda = \lambda_1$.

%\item \textbf{Aperiodicity under real spectrum.}  
%If all eigenvalues of $\mathcal{A}$ are real and simple, then the solution trajectories of the system are non-periodic.

\item \textbf{Asymptotic convergence.}  
Suppose that $\mathcal{A}$ has a unique simple real rightmost eigenvalue $\lambda_1$ with associated eigenmatrix $V_1 \in \mathcal{B}$. If the initial condition $X(0) = X_0 \in \mathcal{B}$ satisfies $\langle X_0, V_1 \rangle \neq 0$, then the solution converges to the dominant eigenmatrix, up to its sign:  
\[
\lim_{t \to \infty} X(t) = \pm V_1.
\]
Note that, by the Perron-Frobenius theorem, in the nonnegative case, the rightmost eigenvalue is always real.
\end{enumerate}
%\begin{remark}
%\end{remark}
 
Therefore, integrating the system \eqref{Ode} for a sufficiently long time, we find the desired eigenmatrix as a stationary point. 
\begin{remark} 
\label{ray_rem}
    If $\cA$ is self-adjoint, \eqref{Ode} is a gradient system for the Rayleigh quotient $\rho(X) = \langle X, \cA(X) \rangle$. The maximum eigenvalue is obtained by maximizing $\rho(X)$.
\end{remark}
The low-rank approximation of the rightmost eigenmatrix proposed in \cite{guglielmi2021computing} is based on a dynamical low-rank approximation \cite{koch2007dynamical} for the evolution equation \eqref{Ode}.
For a fixed value of the rank $r \leq n$, the solution is represented in a factorized form as
\begin{equation}
\label{usv}
X(t) = U(t) S(t) V(t)^\top,
\end{equation}
where \( U(t), V(t) \in \mathbb{R}^{n \times r} \) have orthonormal columns, and \( S(t) \in \mathbb{R}^{r \times r} \). By constraining the dynamics to the manifold of matrices of rank at most \( r \), this method allows  an efficient integration of the system in time.

A particularly effective strategy to compute the factors over time is the {Projector-Splitting Integrator (PSI)} \cite{lubich2014projector}, which updates \( U(t) \), \( S(t) \), and \( V(t) \) sequentially in a way that preserves the rank constraint and ensures numerical stability. In \cite{guglielmi2021computing}, a suitable variant of PSI is used to integrate the system \eqref{Ode}  in time, until the solution converges to a steady state, which provides a rank-\( r \) approximation of the dominant eigenmatrix associated with the rightmost eigenvalue of the operator \( \mathcal{A} \).

\begin{remark}
Using the results of dynamical low-rank theory, see, e.g., \cite{koch2007dynamical, lubich2014projector}, 
it is possible to analytically define the projected system onto the tangent space of the manifold \( \mathcal{M}_r \) of matrices of rank at most \( r \). In the case where \( \mathcal{A} \) is self-adjoint, the projected system turns out to be a gradient flow for the objective function \( \rho(X(t)) \), which enables a complete qualitative analysis, including equilibrium structure and asymptotic behavior. 

In contrast, for a general (non-symmetric) operator \( \mathcal{A} \), the projected system becomes fully nonlinear and does not retain a gradient structure, making its analytical study substantially more difficult. The characterization of stationary points and their long-term behavior in this setting remains an open and challenging problem, see \cite{guglielmi2021computing}.
\end{remark}

%Remark.
%I thought that the dynamics (in the generic case) perhaps could be improved by working with adaptive rank. What I mean is that, with low rank dynamics techniques we could integrate at higher rank in the beginning and then gradually decrease as we gradually approach the stationary point. This could improve the approximation, which in general, even at fixed rank is definitely good.

%\begin{remark}
%I thought that the dynamics (in the generic case) perhaps could be improved by working with adaptive rank. What I mean is that, with low rank dynamics techniques we could integrate at higher rank in the beginning and then gradually decrease as we gradually approach the stationary point. This could improve the approximation, which in general, even at fixed rank is definitely good.
%\end{remark}

\subsection{Taking nonnegativity into account} \label{sec:ouralgo}

%As mentioned in the introduction, %based on Perron's Theorem, 
%How appropriate, 
Now we would like to preserve the nonnegativity of the low-rank eigenmatrix. 
The approach described in the previous section is not trivially adaptable, since the dynamical low-rank approximation is based on the decomposition \eqref{usv}, which involves orthogonality conditions. For simplicity of the presentation, we will omit the argument $t$ when it is clear from the context.   
% Considering what has been said for the dynamical  system \eqref{Ode}, we can see it as finding $\dot{X}$ such that
% \[  
% \max_{\|X\| = 1} \langle \dot{X},\cA(X) \rangle  
% \] 
Let $$ \mathcal{G}(X) = \cA(X)-\rho(X) X.$$ 
Our idea is to interpret the differential equation \eqref{Ode} as the problem of choosing the direction of instantaneous variation $\dot{X}$ that maximally aligns with $\mathcal{G}(X)$, that is, 
\begin{equation}
\label{theproblem}
\dot{X} = \argmax_{ Z \in \mathbb{R}^{n \times n}} 
\langle Z,\mathcal{G}(X) \rangle.
\end{equation}
As discussed in the previous section, although \eqref{Ode} is a gradient system only when the operator is symmetric, the flow still drives the solution toward the desired eigenmatrix also in the general (non-symmetric) case.
This formulation will be useful in incorporating the low-rank  and non-negativity constraints.

Similarly as in~\cite{guglielmi2020efficient}, for a fixed rank $r$, we parameterize the solution as 
$$
X = UV^\top,
$$
with $U, V \in \R^{n \times r}$. 
To obtain the nonnegativity of $X$, we constrain the matrices $U(t)$ and $V(t)$ to be nonnegative and, to maintain nonnegativity, we need to impose that 
$\dot{U}_{i,j}(t) \geq 0$ (resp.\ $\dot{V}_{i,j}(t) \geq 0$) for all $(i,j)$ such that $U_{i,j}(t) = 0$ (resp.\ $V_{i,j}(t) = 0$).   
To do so, for a nonnegative matrix $W \in \mathbb{R}^{n\times r}$, we define: 
\begin{equation}
    E_0(W) := \{(i,j) : W_{ij} = 0\}, 
\end{equation}
and, for a matrix $Z$ of the same dimension as $W$, 
\begin{equation}
   P_W(Z) := 
    \begin{cases} 
        Z_{ij}, & \text{if } (i,j) \in E_0(W) , \\
        0, & \text{otherwise}
    \end{cases} 
\end{equation}
which is the orthogonal projection onto the set of $n$-by-$r$ matrices with zero entries not in $E_0(W)$. 
%We denote, for short, $P_W^+ = P_{E_0(W)} $.\\
We say that $U(t)$, $V(t)$ are feasible paths if $U(t) \geq 0$ and $V(t) \geq 0$ for all $t \geq 0$. 
For every path $U(t) \in \mathbb{R}^{n\times r}$ and $V(t) \in \mathbb{R}^{n\times r}$, the matrices $Z_U, Z_V$ are the derivatives of some path of \emph{feasible} matrices passing through $(U, V)$ if and only if the following two conditions are satisfied: $P_U (Z_U) \geq 0$ and  $P_V (Z_V) \geq 0$, because zero entries in $U$ and $V$ cannot be made smaller, which requires $(Z_U)_{ij} \geq 0$ when $U(i,j) = 0$, and similarly for $V$.

% \begin{align*}
%  \text{i)}  \;\; P_U^+ Z_U &\geq 0,   \nonumber \\
%  \text{ii)}  \;\; P_V^+ Z_V &\geq 0. \nonumber
% \end{align*}
%\ngc{What are $P_U^+$ and $P_V^+$ ?}

We aim to define evolution equations for \(U(t)\) and \(V(t)\) such that the composed matrix \(X(t) = U(t)V(t)^\top\) follows the same dynamics as \(\dot{X} = G(X)\), at least in the manifold of nonnegative rank-\(r\) matrices.
Let us apply the Leibniz rule 
\begin{equation*}
\frac{d}{dt} (U V^\top) = \dot{U} V^\top + U \dot{V}^\top.
\end{equation*}
and let  
\begin{equation*}
\mathcal{G}_1 =  \cA(UV^\top )  - \langle UV^\top,   \cA(UV^\top) \rangle UV^\top.
\end{equation*}
Denoting $\dot{U}$ by $Z_U$ and $\dot{V}$ by $Z_V$, we consider the optimization problem: 
\begin{equation*}
\max_{Z_U, Z_V} \langle Z_U V^\top + U Z_V^\top, \mathcal{G}_1 \rangle \; \text{ such that } \; 
P_U (Z_U) \geq 0,  P_V (Z_V) \geq 0, 
\| Z_U \|_F = 1 \text{\;and\;} \| Z_V \|_F = 1.
\end{equation*}
Since we are searching for directions, we introduced normalization conditions. 
Using the linearity of the inner product, replacing the max by a min and observing that both variables do not interact, 
we rewrite the above problem as two independent problems: 
% \begin{equation*}
% \max \langle Z_U,\mathcal{G}_1 V \rangle + \langle Z_V, \mathcal{G}_1^\top U \rangle \quad \text{subject to i) and ii) and} \| Z_U \|_F = 1 \text{\;and\;} \| Z_V \|_F = 1.
% \end{equation*}
% Since we are optimizing a sum of functions with different independent variables, we split the problem in optimizing each function separately, considering the two sub-problems:
% %\ngc{Are the two problems below solved independently? This is unclear since i) and ii) are for both $U$ and $V$.} 
% \begin{align}
% \max_{Z_U} \langle Z_U, \mathcal{G}_1 V \rangle & \quad \text{subject to i)  and} \langle Z_U, Z_U \rangle = 1, \label{maxU}\\
% \max_{Z_V} \langle Z_V, \mathcal{G}_1^\top U \rangle & \quad \text{subject to ii) and  } \langle Z_V, Z_V \rangle = 1. \label{maxV}
% \end{align}
% The solutions $Z_U$ and $Z_V$ are mutually proportional with positive coefficients, so we choose the direction such that $\|Z_U\|_F = 1$ and $\|Z_V\|_F = 1$.
%Now, let us formulate the problem equivalently as 
\begin{align}
\min_{Z_U} \langle Z_U, -\mathcal{G}_1 V \rangle & \quad \text{ such that } \quad P_U(Z_U) \geq 0  \text{ and }  \langle Z_U, Z_U \rangle = 1, \label{maxU} \\
\min_{Z_V} \langle Z_V, -\mathcal{G}_1^\top U \rangle & \quad \text{ such that } \quad  P_V(Z_V) \geq 0 \text{ and }  \langle Z_V, Z_V \rangle = 1. \label{maxV}
\end{align}
Since these two problems are essentially the same, let us focus on \eqref{maxU}. 
Similarly to \cite{guglielmi2021computing}, let us reformulate it as a quadratic optimization problem with linear constraints. 
%problem~\eqref{maxU} has a quadratic constraint. 
Assume there exists a strict descent direction exists, that is, there exist $Z_U$ such that $\langle Z_U, -\mathcal{G}_1V \rangle < 0$, under the constraint $P_U(Z_U) \geq 0$. 
This is guaranteed as long as $P_U(\mathcal{G}_1V) \neq 0$. 
%there exists an entry $(i,j)$ such that $(\mathcal{G}_1V)_{i,j} \neq 0$ and $(i,j) \notin E_0(W)$, or $(\mathcal{G}_1V)_{i,j} > 0$ and $(i,j) \in E_0(W)$. 
Given such a descent direction, $Z_U$, 
we can rescale it such that $\langle \alpha Z_U, -\mathcal{G}_1V \rangle = -1$, with the scalar $\alpha = \langle \alpha Z_U, \mathcal{G}_1V \rangle^{-1}$. 
%We now formulate a quadratic optimization problem with linear constraints, which is equivalent in the sense that it yields the same descent direction, provided that a strict descent direction exists, that is, such that $\langle Z_U, -\mathcal{G}_1V \rangle < 0$, under the constraint $P_U^+ Z_U \geq 0$. This is based on the fact that when
%$\langle Z_U, -\mathcal{G}_1 V \rangle < 0$, there exists a scaling factor $\alpha > 0$ such that $\langle \alpha Z_U, -\mathcal{G}_1V \rangle = -1$. 
Hence under the existence of a descent direction, the problem~\eqref{maxU} is equivalent to %\ngc{you said you would focus only on (8): why bring back (9) already?}
\begin{align}
&\min_{Z_U} \langle Z_U, Z_U \rangle  \quad \text{ such that } \quad P_U(Z_U) \geq 0  \text{ and }  \langle Z_U,  \mathcal{G}_1 V \rangle = 1 \label{maxU1} 
% \\
% &\min_{Z_V} \langle Z_V, Z_V \rangle  \quad \text{ such that } \quad  P_V^+ Z_V \geq 0 \text{ and } \langle Z_V, \mathcal{G}_1^\top U \rangle = 1.\label{maxV1}
\end{align} 
%As the constraints act independently on each matrix entry and there is no interaction among them, the result can be derived by considering the scalar (vector) case without loss of generality. 
% For a nonnegative matrix $W \in \mathbb{R}^{n\times r}$, we define: 
% \begin{equation}
%     E_0(W) := \{(i,j) : W_{ij} = 0\}, 
% \end{equation}
% and, for a matrix $Z$ of the same dimension as $W$, 
Defining 
\begin{equation}
   P_W^+(Z) := 
    \begin{cases} 
        Z_{ij}, & \text{if } W_{ij} > 0, \\
        \max(0, Z_{ij}), & \text{if } W_{ij} = 0, 
    \end{cases} 
\end{equation}
for any $Z \in \mathbb{R}^{n\times r}$, 
the problem~\eqref{maxU1} can be solved in closed form, 
with optimal solution $Z_U^* = \frac{P_U^+( \mathcal{G}_1 V)}{\langle \mathcal{G}_1 V, P_U^+( \mathcal{G}_1 V)\rangle}$, 
as a consequence of the following lemma (in vector form).  
\begin{lemma}
Let \( a \in \mathbb{R}^n \) and let \( \Omega \subset \{1, \dots, n\} \) be an index set such that there exists at least one index \( i \notin \Omega \) with \( a_i > 0 \). Consider the optimization problem
\begin{equation} \label{eq:opt_problem}
    \min_{x \in \mathbb{R}^n} \|x\|_2 \quad \text{subject to} \quad \langle a, x \rangle = 1, \quad x_i = 0 \text{ for all } i \in \Omega.
\end{equation}
Then the unique optimal solution is given by
\[
x^* = \frac{P_{\mathcal{I}}(a)}{\langle a, P_{\mathcal{I}}(a) \rangle},
\]
where \( \mathcal{I} := \{1, \dots, n\} \setminus \Omega \), and \( P_{\mathcal{I}}(a) \in \mathbb{R}^n \) is the projection of \( a \) onto the coordinates indexed by \( \mathcal{I} \), that is,
\[
\big(P_{\mathcal{I}}(a)\big)_i = \begin{cases}
a_i, & \text{if } i \in \mathcal{I}, \\[1mm]
0, & \text{if } i \in \Omega.
\end{cases}
\]
\end{lemma}

\begin{proof}
Problem~\eqref{eq:opt_problem} is convex, with a strictly convex objective and affine equality constraints. Therefore, the Karush-Kuhn-Tucker (KKT) conditions are both necessary and sufficient for optimality. 
We introduce the Lagrangian:
\[
\mathcal{L}(x, \lambda, \mu) = \frac{1}{2} \|x\|_2^2 - \lambda \big( \langle a, x \rangle - 1 \big) - \sum_{i \in \Omega} \mu_i x_i,
\]
where \( \lambda \in \mathbb{R} \) is the multiplier for the equality constraint \( \langle a, x \rangle = 1 \), and \( \mu_i \in \mathbb{R} \) are multipliers for the constraints \( x_i = 0 \), for \( i \in \Omega \). 
The KKT stationarity condition reads:
\[
\nabla_x \mathcal{L}(x, \lambda, \mu) = x - \lambda a - \sum_{i \in \Omega} \mu_i e_i = 0,
\]
which implies
%\[
$x = \lambda a + \sum_{i \in \Omega} \mu_i e_i$.
%\]
Evaluating this componentwise: 
\begin{itemize}
\item For \( i \in \Omega \), the constraint \( x_i = 0 \) yields:
%\[
$0 = \lambda a_i + \mu_i \quad \Rightarrow \quad \mu_i = -\lambda a_i$.
%\]
\item For \( i \in \mathcal{I} \), we have \( \mu_i = 0 \), so \( x_i = \lambda a_i \).
\end{itemize}
Hence, the optimal solution is 
%\[
$x^* = \lambda P_{\mathcal{I}}(a)$. 
%\]
To enforce the constraint \( \langle a, x \rangle = 1 \), we compute:
\[
\langle a, x^* \rangle = \lambda \langle a, P_{\mathcal{I}}(a) \rangle = 1 \quad \Rightarrow \quad \lambda = \frac{1}{\langle a, P_{\mathcal{I}}(a) \rangle}.
\]
Substituting back gives the optimal solution, 
%\[
$x^* = \frac{P_{\mathcal{I}}(a)}{\langle a, P_{\mathcal{I}}(a) \rangle}$, 
%\]
which is unique since the problem is strictly convex with affine constraints, the solution is unique.

\end{proof}

% \begin{remark}
% The result extends directly to the matrix-valued case considered in our setting, as the optimization decouples across entries due to the separable structure of the constraints. Each column (or row) can thus be treated independently using the above result.
% \end{remark}

As a consequence, we have that on an interval where $E_0(U(t))$ and $E_0(V(t))$ do not change, the differential system for the factors becomes %(we omit the argument $t$ for conciseness):
\begin{equation}
    \dot{U} = P_U^+ (\mathcal{G}_1 V) 
    \quad \text{ and } \quad 
    \dot{V} = P_V^+ (\mathcal{G}_1^\top U). \label{UdotVdot}
\end{equation}

\subsection{Low-rank integrator and computational cost} \label{sec:pseudocode}

In this section, we present the discrete version of the nonnegative low-rank integrator, highlighting its numerical formulation and computational cost. We extend the method by including a backtracking strategy for adaptive step-size control, which enhances stability and ensures consistency with the theoretical dynamics.
 The method solves the differential equation~\eqref{UdotVdot} for the  factors $U \in \mathbb{R}^{n \times r}$ and $V \in \mathbb{R}^{m \times r}$. 
 This differential equation is integrated numerically; in practice, we use an explicit Euler discretization.
Clearly, this is not the only possible choice, but in this case it is convenient, as we are interested in stationary points of the system.

\begin{algorithm}[ht!]
\caption{Nonnegative Low-Rank Integrator (RNeg)\label{alg1}} 
\begin{algorithmic}[1]
\Require 
\begin{itemize}[noitemsep, topsep=0pt]
  \item Matrix operator $\cA = \cA(U,V) \in  {R}_+^{m \times n}$ 
  %\( A_i \in \mathbb{R}_+^{m \times m},\ B_i \in \mathbb{R}_+^{n \times n} \) for \( i = 1, \dots, k \), and weights \( \alpha \in \mathbb{R}^k \) \ngc{I think we should be more general here, just mentioning an operation that computes $\mathcal A(U,V)$ }
  \item Initial low-rank factors \( U \in \mathbb{R}_+^{m \times r},\ V \in \mathbb{R}_+^{n \times r} \)
  \item Step size \( h > 0 \), tolerance \texttt{tol}, maximum iterations \texttt{nmax}
\end{itemize}

\Ensure 
\begin{itemize}[noitemsep, topsep=0pt]
  \item Factors \( (U, V) \) such that \( X = U V^\top \in \mathbb{R}_+^{m \times n} \) is an approximate nonnegative low-rank solution to~\eqref{Ode}
  \item Approximate eigenvalue \( \lambda \)
  \item Convergence measures \( d_U, d_V \)
\end{itemize}

\State $d_U = \infty$, $d_V = \infty$,\quad $step = 0$
\While{$\min(d_U, d_V) >$ \texttt{tol} \textbf{and} $step <$ \texttt{nmax}}
    \State $step = step + 1$
    \State $U_{\text{prev}} = U,\; V_{\text{prev}} = V$ 
    \hfill \Comment{\( \mathcal{O}(mr+nr) \)}
    
    \State Compute \( F = \mathcal{A}(U, V) \) 
    \hfill \Comment{\( \mathcal{O}\big(k(m^2r + n^2r + mnr)\big) \)}
    
    \State \( \mathrm{UtU} = U^\top U,\quad \mathrm{VtV} = V^\top V \) 
    \hfill \Comment{\( \mathcal{O}(mr^2 + nr^2) \)}
    
    \State \( \lambda = \mathrm{Re}(\mathrm{trace}(F^\top U V^\top)) \) 
    \hfill \Comment{\( \mathcal{O}(mnr) \)}
    
    \State \( G_U = F V - \lambda U \mathrm{VtV} \) 
    \hfill \Comment{\( \mathcal{O}(mnr + mr^2) \)}
    
    \State \( G_V = F^\top U - \lambda V \mathrm{UtU} \) 
    \hfill \Comment{\( \mathcal{O}(mnr + nr^2) \)}

    \State Compute admissible step size $h_{\max}$ to preserve nonnegativity
    \[
      h_{\max} = \min \Big\{ 
      \min_{(i,j): U_{ij}=0,\ G_{U,ij}<0} \tfrac{-U_{ij}}{G_{U,ij}}, \ 
      \min_{(i,j): V_{ij}=0,\ G_{V,ij}<0} \tfrac{-V_{ij}}{G_{V,ij}}
      \Big\}
    \]
    \State $h \gets \min(h, h_{\max})$
    \hfill \Comment{\( \mathcal{O}(mr + nr) \)}
    
    \State \( U = U + h \cdot \mathcal{P}^+_UG_U,\quad V = V + h \cdot \mathcal{P}^+_VG_V \) 
    \hfill \Comment{\( \mathcal{O}(mr + nr) \)}
    
    \State \( U = \max(U, 0),\quad V = \max(V, 0) \) 
    \hfill \Comment{Entrywise, \( \mathcal{O}(mr + nr) \)}
    
    \State \( s = \sqrt{\mathrm{trace}(\mathrm{UtU} \cdot \mathrm{VtV})} \) 
    \hfill \Comment{\( \mathcal{O}(r^2) \)}
    
    \State \( U = U / \sqrt{s},\quad V = V / \sqrt{s} \) 
    \hfill \Comment{\( \mathcal{O}(mr + nr) \)}
    
    \State \( d_U = \| U - U_{\text{prev}} \|_F \) 
    \hfill \Comment{\( \mathcal{O}(mr) \)}
    \State \( d_V = \| V - V_{\text{prev}} \|_F \) 
    \hfill \Comment{\( \mathcal{O}(nr) \)}
\EndWhile

\State \Return \( U, V, \lambda, d_U, d_V \)
\end{algorithmic}
\end{algorithm}

%\begin{algorithm}[H]\label{alg:projection}
%\caption{Projection onto the admissible cone $\mathcal{P}^+$ \ngc{I don't think we need an algorithm to define this projection. I can be defined in the text; see \eqref{projUplus}.}}
%\begin{algorithmic}[1]
%\Require 
%\begin{itemize}[noitemsep, topsep=0pt]
%  \item Current iterate \( W \in \mathbb{R}^{p \times r} \) (e.g.\ \(U\) or \(V\))
 % \item Candidate update \( Z \in \mathbb{R}^{p \times r} \) (e.g.\ gradient step)
%\end{itemize}
%\Ensure 
%\begin{itemize}[noitemsep, topsep=0pt]
%  \item Projected update \( \widehat{Z} = \mathcal{P}^+_{W}(Z) \)
%\end{itemize}

%\For{each entry $(i,j)$} \hfill \Comment{\( \mathcal{O}(pr) \)}
 %   \If{$W_{ij} > 0$} 
  %      \State $\widehat{Z}_{ij} = Z_{ij}$
  %  \Else
   %     \State $\widehat{Z}_{ij} = \max(0, Z_{ij})$
%    \EndIf
%\EndFor
%\State \Return $\widehat{Z}$
%\end{algorithmic}
%\end{algorithm}

For all applications considered in this work, the operator $\cA(U,V) $ can be evaluated directly in factorized form, leading to efficient implementations. 
For example, for Markov grids, the matrix field \( \mathcal{A}(U, V) \) is never applied to the full matrix \( X = UV^\top \), but is instead computed efficiently in factorized form: 
\[
F = \mathcal{A}(U, V) = \sum_{i=1}^k \alpha_i (A_i^\top U)(B_i V)^\top,
\]
is computed at each iteration without forming the full matrix \( X \in \mathbb{R}_+^{m \times n} \). 
% Each term involves:
% \begin{itemize}
%   \item Multiplying \( A_i^\top \in \mathbb{R}^{m \times m} \) with \( U \in \mathbb{R}^{m \times r} \), costing \( \mathcal{O}(m^2 r) \),
%   \item Multiplying \( B_i \in \mathbb{R}^{n \times n} \) with \( V \in \mathbb{R}^{n \times r} \), costing \( \mathcal{O}(n^2 r) \),
%   \item Multiplying the resulting \( m \times r \) and \( n \times r \) matrices to form \( (A_i^\top U)(B_i V)^\top \), costing \( \mathcal{O}(m n r) \).
% \end{itemize} 
% Therefore, 
The total cost for computing \( F \) is $\mathcal{O}\big( k (m^2 r + n^2 r + m n r) \big)$ operations.  
Additional computations at each iteration include:
\begin{itemize}
  \item Computing \( U^\top U \) and \( V^\top V \) at cost \( \mathcal{O}((m + n) r^2) \),
  \item Computing  \( G_U \) and \( G_V \) at cost \( \mathcal{O}(m n r + (m + n) r^2) \),
  \item Updating \( U \) and \( V \) and enforcing element-wise nonnegativity at cost \( \mathcal{O}((m + n) r) \),
  \item Computing the normalization scalar \( s \) at cost \( \mathcal{O}(r^2) \),
  \item Normalizing \( U \) and \( V \) at cost \( \mathcal{O}((m + n) r) \).
\end{itemize}

The overall dominant cost per iteration is thus
\[
\mathcal{O}\big( k (m^2 r + n^2 r + m n r) \big),
\]
which is substantially cheaper than full-rank methods, since the rank \( r \) is small compared to \( m \) and \( n \).
Similarly to the Markov grid case, the operator action in the PDE experiments is evaluated directly in low-rank form;  see Section \ref{PDE_experiments}.

A step-size control to maintain nonnegativity ensures consistency with the continuous-time theory: each discrete update is performed only as long as the zero pattern of the factor matrices does not change. Projections onto the zero components are applied as a safeguard, keeping the support consistent with the theoretical dynamics; 
see \eqref{UdotVdot}.

At each iteration, we employ a backtracking line search strategy.
Specifically, the step size is accepted only if the Frobenius norms of the {projected directions}, $P_U^+(G_U)$ and $P_V^+ (G_V)$, decrease, where the projection is defined entrywise according to the support of the current computed iterates. This projection ensures that only the active components—those corresponding to positive entries—contribute to descent, reflecting the constraint of the dynamics to the nonnegative orthant. By enforcing this condition, the algorithm maintains consistency with the continuous differential flow along the feasible directions. If the condition is not satisfied, the step is rejected and reduced by a factor $0<\beta_{rej}<1$; conversely, upon acceptance, the step size is allowed to grow moderately, by a multiplicative factor $\beta_{acc}>1$.

\section{Numerical Experiments} \label{sec:numexp}

In this section, we present numerical experiments  to assess the performance of the proposed method, Algorithm~\ref{alg1}. 
We consider two classes of test problems: Markov grid models (see Section~\ref{sec:grids}) and discretized PDEs (see Section~\ref{sec:metzler}).
For both settings, we compare our approach with several baseline methods in terms of computational cost, accuracy of the computed eigenpairs, and preservation of nonnegativity.
The experiments aim to highlight the effectiveness of the proposed algorithm across different problem structures and to illustrate its practical behavior in representative large-scale scenarios.
Numerical experiments were performed in MATLAB 2024a on a MacBook Air with Apple M3 (8 cores) and 16 GB RAM. 
The code is available from \url{https://gitlab.com/ngillis/nonnegative-low-rank-leading-eigenmatrix/}.

\paragraph{Compared algorithms}

We compute a reference solution for each experiment.
In particular, for the Markov grid tests we compute the dominant eigenpair $(\lambda^*, X^*)$
using the Power method applied to the full problem, with tolerance $10^{-8}$. 
The reason we choose the power method for the Makov grids is that \texttt{EIGS} was not able to compute accurate solutions in this setting, because it has to construct explicitly the large transition matrix $P \in \mathbb{R}^{m n \times m n}$ which appears to lead to numerical instability; see~\eqref{eq:mattovec} and the discussion around it. 
%As an additional reference, we also employ the \texttt{EIGS} function of MATLAB applied to the full (equivalent) Kronecker matrix.
For the experiments involving operators arising from the discretization of specific PDEs,
the reference solution is not computed via the Power method. 
In this case, we rely exclusively on the \texttt{EIGS} function, since we seek the eigenvalue
with largest real part, which does not necessarily coincide with the spectral radius. 
Low-rank approximations of the eigenmatrix $X^*$ are then obtained via truncated SVD and NMF, which we denote by
Power+SVD and Power+NMF for the Markov grid and EIGS+SVD and EIGS+NMF for the PDEs, respectively. To compute the NMF, we use \url{https://gitlab.com/ngillis/nmfbook/}. 
We further compare against the low-rank integrator PSI proposed in~\cite{guglielmi2021computing}. 
Finally, we consider our proposed approach, Algorithm~\ref{alg1}, dubbed RNeg, which computes the dominant
eigenmatrix by evolving a continuous-time dynamical system constrained to have fixed (low) rank
and nonnegativity.

\paragraph{Performance metrics} We report the computational time, in seconds, required by each method. 
For a linear operator $\mathcal A$, we compute the groundtruth (possibly high rank) solution with either the power method for Markov grids or with EIGS for the PDEs, and denote it $(X^*,\lambda^*)$.  
Given a solution $(X,\lambda)$ computed by one of the considered  methods, we will report the relative error %RelErr \ngc{Can you define it?} and the best scaling errors: 
\[
\text{RelErr} = \min_\alpha \frac{\|\alpha X-X^*\|_F}{\|X^*\|_F}. 
\] 
The reason we scale $X$ is that some methods normalize $X$ in the $\ell_2$ norm (e.g., \texttt{EIGS}), while $X^*$ might have a different norm (e.g., for Markov grids with the power method, it has $\ell_1$ norm equal to one). 
We will also report the values of $\|\mathcal{A}(X) - \lambda X\|_F$, $|\lambda^* - \lambda|$, 
along with the number of negative entries in $X$. 
For all methods, the approximated matrices $X$ are normalized in Frobenius norm prior to computing the approximate eigenvalue $\lambda$, the residual $\|\mathcal A(X) - \lambda X\|_F$, and the number of negative entries, ensuring that the comparisons across different algorithms are consistent and independent of scaling.

\subsection{Markov grids} 

In this experiment, we consider a class of structured test problems based on Markov grids.

\paragraph{Synthetic Markov grids} We construct block-structured operators from pairs of row-stochastic matrices 
\(\{(A_i,B_i)\}_{i=1}^{t+1}\) combined with positive weights \(\{\alpha_i\}_{i=1}^{t+1}\) such that 
\(\sum_{i=1}^{t+1} \alpha_i = 1\). Recall that the action of the operator on a matrix \(X\) is defined as 
\begin{equation}
    \mathcal{A}(X) = \sum_{i=1}^{t+1} \alpha_i \, A_i^\top X B_i .
\end{equation}  
The first $t$ terms are block sub-grids of size $N_i \times N_i$, and the Markov grid has size $m = n = \sum_{i=1}^t N_t$. 
Each matrix $A_i$ and $B_i$ has nonzero entries 
only within the rows and columns corresponding to block~$i$. 
Each row of the nonzero block matrices are generated using the Dirichlet distribution with uniform parameter equal to one\footnote{It generates vectors uniformly at random on the unit simplex, $\{x \ | \ x \geq 0, \sum_i x_i = 1\}$.}, making the $A_i$'s and $B_i$'s row stochastic, so that the operator preserves the probabilistic interpretation of the process; see Section~\ref{sec:grids}.  
The last term ($i=t+1$) is a dense, fully-connected component, also generated using the Dirichlet distribution of uniform parameter one, with the corresponding parameter $\alpha_{t+1} = \delta$, while $\alpha_{1:t}$ is generated using the Dirichlet distribution with uniform parameter one, and multiplied by $(1-\delta)$.

This setting is particularly interesting; 
while the full-rank stationary matrix is nonnegative, its low-rank approximations  typically contain negative entries, similarly as in Example~\ref{ex:neglowrank33}. %which are inconsistent with the probabilistic structure of the problem. 
%This motivates the use of alternative approaches such as (NMF or our proposed low-rank integrators, which are able to provide approximations that better respect the underlying structure.
Although these matrices are not low-rank, 
we consider them as examples for our controlled benchmark. 
It allows us to test the ability of low-rank integrators and other methods to preserve 
important structural properties such as non-negativity and block-wise interactions, 
even when the underlying stationary matrix is full rank.

\paragraph{Experiment 1: block Markov grids}  

The first experiment considers a block Markov grid of size $n = 50$ with $t=n$ blocks 
of size one ($N = [1,\dots,1]$) and a fully connected component weighted by $\delta = 0.2$.  
%In our numerical experiment, we generate an operator of size $n=50$. The full-rank solution $X^*$ is non-negative by construction and it corresponds to the eigenvalue $1$.
For all low-rank methods, we set the rank to $r=10$. 
Table~\ref{tab:merged_markov50} reports the  average results over 10 runs of the tested algorithms. 
%\ngc{Explain what we } 
%\ngc{I think we can merge Tables 1 and 2, removing the parameters $n$,$r$,$\delta$, and merging columns using $\pm$ for the standard deviation, e.g., $1.5 \pm 0.05$. I can do it if you agree.} 
%Specifically, we tested: the \textit{Power Method}, \textit{SVD}, \textit{Frobenius NMF} (full-rank and low-rank after SVD reduction), the \textit{Projector Splitting Integrator (PSI)}, as well as our method (\textit{NNeg}) and the direct \textit{EIGS} approach.  
%\ngc{Isn't our approach called RNeg? Why this name by the way?} 
%Metrics considered include \textit{computation time}, \textit{relative and best-scaling errors}, \textit{eigenpair residuals}, and the \textit{number of negative entries}.  
We observe that, due to the sparsity of the grid, the Power Method—as well as Power+SVD and Power+NMF—is extremely efficient in terms of computation time for obtaining the low-rank solution, making iterative integration methods such as RNeg and PSI less competitive. 
In terms of accuracy, RNeg is comparable to Power+NMF. 
By construction, RNeg and  Power+NM preserve nonnegativity, whereas Power+SVD and PSI produce a significant number of negative entries. 
%Based on the relative error (RelErr), the best performance is achieved by Power+SVD and Power+NMF. 
When considering the  relative error (RelErr) and the residual norm $\|AX - \lambda X\|_F$, 
all methods show essentially similar performance. Finally, RNeg yields the most accurate eigenvalue.
In summary, RNeg provides a good trade-off between accuracy and nonnegativity constraints, even if it is not competitive in computational time compared to the Power Method on sparse grids. A topic of further research would be to make RNeg faster.

\begin{table}[h!]
\centering
\footnotesize
\setlength{\tabcolsep}{4pt}
\begin{tabular}{l c c c c c}
\toprule
Method
& Time $(\times 10^{-2})$
& RelErr $(\times 10^{-1})$
& $\lVert AX-\lambda X\rVert_F$
& $|\lambda-\lambda_{PM}| \; (\times 10^{-3})$
& \#neg \\
\midrule
Power
& $2.35 \pm 1.23$
& $< 10^{-14}$
& $< 10^{-14}$
& $< 10^{-14}$
& $0$ \\ 
\midrule 

Power+SVD
& $2.46 \pm 1.21$
& $8.89 \pm 0.00003$
& $1.925 \pm 0.003$
& $9.61 \pm 1.54$
& $838 \pm 23$ \\

Power+NMF
& $2.84 \pm 1.40$
& $8.89 \pm 0.00004$
& $1.957 \pm 0.006$
& $6.18 \pm 2.86$
& $0$ \\

PSI
& $(2.96 \pm 0.03)\,10^{4}$
& $8.89 \pm 0.00003$
& $1.929 \pm 0.003$
& $6.70 \pm 1.33$
& $1089 \pm 398$ \\

RNeg
& $(6.64 \pm 0.10)\,10^{4}$
& $8.89 \pm 0.00003$
& $1.940 \pm 0.003$
& $2.30 \pm 1.33$
& $0$ \\

\bottomrule
\end{tabular}
\caption{Computational results for $n=50$, $r=10$, $\delta=0.2$ over $10$ trials. Each entry reports mean $\pm$ standard deviation.}
\label{tab:merged_markov50}
\end{table}

\paragraph{Experiment 2: random Markov grids}

In addition to the structured block grids, we also consider fully random  Markov grids. 
Each matrix $A_i$ and $B_i$ is generated using 
\texttt{sprand} to create random entries, with density $\sigma = 0.9$. The matrices are then 
normalized so that the entries on their rows sum to one. 
The matrices are combined using a set of random convex weights $\alpha$, sampled from a Dirichlet distribution to ensure $\sum_i \alpha_i = 1$.

This setting represents a more unstructured, large-scale problem where the operator lacks the almost block-diagonal pattern of the previous example. The results for $n = 100$ and $r=3$ are shown in Table \ref{merged_markov100}.
In this case, the low-rank approximations of $X^*$ are nonnegative.

\begin{table}[h!]
\footnotesize
\setlength{\tabcolsep}{4pt}
\centering
\begin{tabular}{l c c c c c}
\toprule
Method
& Time $(\times 10^{-2})$
& RelErr
& $\lVert AX-\lambda X\rVert_F$
& $|\lambda-\lambda_{PM}|$
& \#neg \\
\midrule
Power
& $0.756 \pm 0.608$
& $< 10^{-14}$
& $< 10^{-13}$
& $< 10^{-14}$
& $0$ \\
\midrule
Power+SVD
& $1.10 \pm 0.78$
& $< 10^{-14}$
& $< 10^{-14}$
& $< 10^{-13}$
& $0$ \\

Power+NMF
& $1.91 \pm 0.14$
& $(6.79 \pm 16.2)\,10^{-8}$
& $(6.81 \pm 16.2)\,10^{-8}$
& $(1.63 \pm 4.85)\,10^{-10}$
& $0$ \\

PSI
& $(3.99 \pm 0.12)\,10^{3}$
& $(1.38 \pm 3.70)\,10^{-2}$
& $(1.35 \pm 3.62)\,10^{-2}$
& $(1.32 \pm 4.13)\,10^{-3}$
& $0$ \\

RNeg
& $(1.61 \pm 33.4)\,10^{3}$
& $(8.16 \pm 18.9)\,10^{-6}$
& $(8.78 \pm 20.4)\,10^{-6}$
& $(1.26 \pm 2.94)\,10^{-7}$
& $0$ \\

% EIGS
% & $0.878 \pm 0.144$
% & $(9.21 \pm 1.11)\,10^{-2}$
% & $(9.22 \pm 1.11)\,10^{-2}$
% & $< 10^{-14}$
% & $0$ \\
\bottomrule
\end{tabular}
\caption{Computational results for $n=100$, $r=3$, $\delta=0.10$ over $10$ trials. Each entry reports mean $\pm$ standard deviation.}
\label{merged_markov100}
\end{table}
In this case, RelErr coincide for all methods, which all achieve excellent performance, 
except for PSI, which reaches an accuracy of only the order of $10^{-2}$. Similar results are observed for 
$\|AX - \lambda X\|$ and for the error in the eigenvalue.
This second experiment shows that, when the solution of the problem truly has a low-rank (or nearly low-rank) structure, the performance of our method is comparable to NMF applied to an accurate reference solution (that is, Power+NMF).

%The experiment presented in the previous section, while challenging because of the lack of a genuine low-rank structure in the solution, allows us to assess its robustness, since it performs similarly  than NMF applied to an accurate reference solution (that is, Power+NMF). 

%Clearly, we are referring to the comparison with , as it is the only method—besides RNeg—able to preserve positivity in a low rank setting.

\subsection{Positive dominant eigenmodes in diffusion models}
\label{PDE_experiments}

We consider a class of spatially heterogeneous growth–diffusion models~\cite{cantrell2003spatial} as discussed in Section~\ref{sec:metzler}. 
Specifically, the continuous PDE reads
\[
    \frac{\partial u}{\partial t} = \varepsilon \, \Delta u + r(x,y) \, u, 
    \qquad (x,y) \in \Omega \subset \mathbb{R}^2,
\]
where $\varepsilon>0$ is the diffusion coefficient and $r(x,y)$ is a spatially varying growth rate. 
The diffusion term accounts for movement or mixing, while $r(x,y)$ represents local demographic variations. 
Depending on the boundary conditions, the model can describe closed habitats (Neumann), partially open systems (Robin), or lethal boundaries (Dirichlet).  
Depending on the boundary conditions, the model can describe closed habitats (Neumann), partially open systems (Robin), or lethal boundaries (Dirichlet).  
The asymptotic behavior of $u$ is governed by the {principal eigenpair} $(\lambda_{\max}, u_{\max})$ of the associated elliptic operator. 
The principal eigenvalue $\lambda_{\max}$ (spectral abscissa) determines persistence or extinction, while the corresponding eigenfunction $u_{\max}$ describes the stationary spatial distribution and is  positive within $\Omega$.

More general diffusion models can introduce heterogeneity through spatially varying diffusion coefficients $\varepsilon(x,y)$. 
To numerically study the principal eigenmodes, we discretize the spatial domain $\Omega$ on an $n \times n$ uniform grid. 
Let $Z \in \mathbb{R}^{n \times n}$ denote the matrix of grid values, with entries 
\[
Z_{ij} \approx u(x_i,y_j), \quad i,j = 1,\dots,n.
\] 
Applying finite-difference approximations to the Laplacian and evaluating the growth term on the grid, the continuous PDE is transformed into the semi-discrete system
\begin{equation}
\dot{Z} = \varepsilon \, (A Z + Z A^\top) + \varepsilon_r \, R \circ Z,
\end{equation}
where $A \in \mathbb{R}^{n \times n}$ is the standard finite-difference discretization of the one-dimensional Laplacian with Neumann boundary conditions, 
$\circ$ denotes the Hadamard (elementwise) product, and
\begin{equation}
R = r_0 + \phi \, \psi^\top,
\end{equation}
with separable modulation vectors
\begin{equation}
\phi = \sin(2\pi x), \qquad \psi = \cos(2\pi y),
\end{equation}
where $x = (x_1,\dots,x_n)^\top$ and $y = (y_1,\dots,y_n)^\top$ are the grid coordinates in the $x$ and $y$ directions. 
The matrix $R \in \mathbb{R}^{n \times n}$ encodes the spatial heterogeneity of the growth term over the grid. 
The principal eigenpair of the semi-discrete operator determines the asymptotic behavior of $Z$, analogous to the continuous PDE.

The discretized operator is Metzler, so the Perron--Frobenius theorem ensures the positivity of the dominant eigenmatrix. 
This property is crucial in biological and physical applications, where negative densities are nonphysical. 
Furthermore, in terms of the numerical approximation of the solution of the PDE, the stiffness introduced by fine spatial discretization makes implicit or semi-implicit time integration schemes particularly suitable. 
The knowledge of the operator's spectral abscissa is important to guarantee the conservation of positivity in the numerical solution~\cite{hundsdorfer2003numerical}.

We then compute the dominant eigenpair by applying the proposed low-rank positive algorithm to the discretized operator: 
\[
f(U,V) =
\varepsilon \big( (AU)V^{\top} + U(AV)^{\top} \big)
+
R \circ (UV^{\top}).
\]  
Recall that the reference solution, denoted $X^*$, is obtained using the MATLAB routine \texttt{EIGS}, which computes the dominant eigenpair of the sparse Metzler operator associated with the discretized problem. %The resulting dominant eigenmatrix is normalized to have unit Frobenius norm. 
%We then apply two low-rank integrators — the projector-splitting integrator (PSI) and the proposed method (RNeg) — starting from the same random nonnegative initialization. For comparison, we also consider either a truncated SVD (EIGS+SVD) or a NMF (EIGS+NMF), which provide low-rank approximations of $X_{\mathrm{dom}}$. \ngc{Already explained}
%For each experiment, we provide two tables summarizing the results.
%We report, for every method, the Frobenius error $\|X - X_{\mathrm{dom}}\|_F$, the computational time, the number of negative entries, the computed eigenvalue, its error with respect to the reference solution, and the deviation from being an exact eigenpair, measured as
%\[
%\| \mathcal{A}(X) - \lambda X \|_F .
%\]
%All matrices are normalized to unit Frobenius norm to enable a fair comparison.
A method that preserves the nonnegativity structure of the dominant eigenmatrix 
should exhibit a strictly positive minimum value and no negative entries, in 
agreement with the Perron--Frobenius property of Metzler matrices. 
%The second table lists, for each method, 
For the numerical tests, we consider the values $r_0 = 0.1, \varepsilon = 0.01$ and $\varepsilon_r = 3\pi$. The initial stepsize is $h = 5 \times 10^{-3}$ for both PSI and RNeg.
We emphasize that, for this type of operators,  the first two eigenvalues with largest real parts are very close to each other, while we assumed that the spectral abscissa must be a simple eigenvalue. 
Therefore, this is a particularly challenging case for the methods PSI and  RNeg, but from the results we see that the eigenvalue approximation is always quite good. For RNeg we have good performances in terms of computational time in this case.
Tables~\ref{tab:100r3} and~\ref{tab:100r4} provide the results for dimension $n = 200$ and ranks $3$ and $4$, respectively and Tables~\ref{tab:200r3} and~\ref{tab:200r4} for $n = 200$ and ranks $3$ and $4$, respectively.

\begin{table}[h!]
\centering
\footnotesize
\setlength{\tabcolsep}{4pt}
\begin{tabular}{l c c c c c}
\toprule
Method & Time  & RelErr  & $\lVert AX-\lambda X\rVert_F$ & $|\lambda-\lambda_{max}|$ & \#neg \\
\midrule
EIGS & $0.33$ & $< 10^{-14}$ & $< 10^{-14}$ & $1.3 \times 10^{-12}$ & $0$ \\ \midrule
EIGS+SVD & $0.33$ & $7.5 \times 10^{-3}$ & $1.1 \times 10^{-1}$ & $7.4 \times 10^{-4}$ & $1222$ \\
EIGS+NMF & $0.34$ & $8.8 \times 10^{-3}$ & $7.8 \times 10^{-2}$ & $6.0 \times 10^{-4}$ & $0$ \\
PSI & $5.1$ & $4.4 \times 10^{-2}$ & $7.0 \times 10^{-2}$ & $5.6 \times 10^{-4}$ & $722$ \\
RNeg & $0.23$ & $7.0 \times 10^{-1}$ & $1.0 \times 10^{-1}$ & $1.2 \times 10^{-3}$ & $0$ \\
\bottomrule
\end{tabular}
\caption{Computational results for $n=100$ and $r=3$.}
\label{tab:100r3}
\end{table}

\begin{table}[h!]
\centering
\footnotesize
\setlength{\tabcolsep}{4pt}
\begin{tabular}{l c c c c c}
\toprule
Method & Time  & RelErr  & $\lVert AX-\lambda X\rVert_F$ & $|\lambda-\lambda_{max}|$ & \#neg \\
\midrule
EIGS & $0.46$ & $< 10^{-14}$ &  $1.3 \times 10^{-12}$ & $< 10^{-14}$ & $0$ \\ \midrule
EIGS+SVD & $0.46$ & $4.8 \times 10^{-4}$ & $8.2 \times 10^{-3}$ & $3.8 \times 10^{-6}$ & $378$ \\
EIGS+NMF & $0.49$ & $4.8 \times 10^{-4}$ & $8.3 \times 10^{-3}$ & $3.9 \times 10^{-6}$ & $0$ \\
PSI & $5.6$ & $1.0 \times 10^{-1}$ & $8.2 \times 10^{-3}$ & $4.3 \times 10^{-6}$ & $252$ \\
RNeg & $0.51$ & $7.1 \times 10^{-1}$ & $1.0 \times 10^{-1}$ & $1.2 \times 10^{-3}$ & $0$ \\
\bottomrule
\end{tabular}
\caption{Computational results for $n=100$ and $r=4$.}
\label{tab:100r4}
\end{table}

\begin{table}[h!]
\centering
\footnotesize
\setlength{\tabcolsep}{4pt}
\begin{tabular}{l c c c c c}
\toprule
Method & Time  & RelErr  & $\lVert AX-\lambda X\rVert_F$ & $|\lambda-\lambda_{max}|$ & \#neg \\
\midrule
EIGS & $5.4$ & $< 10^{-14}$ & $4.7 \times 10^{-12}$  & $< 10^{-14}$ & $0$ \\ \midrule
EIGS+SVD & $5.5$ & $7.5 \times 10^{-3}$ & $1.1 \times 10^{-1}$ & $7.5 \times 10^{-4}$ & $4974$ \\
EIGS+NMF & $5.5$ & $8.5 \times 10^{-3}$ & $8.6 \times 10^{-2}$ & $6.5 \times 10^{-4}$ & $0$ \\
PSI & $38$ & $1.0 \times 10^{-1}$ & $6.5 \times 10^{-2}$ & $4.7 \times 10^{-4}$ & $2703$ \\
RNeg & $5.3$ & $7.2 \times 10^{-1}$ & $1.0 \times 10^{-1}$ & $1.2 \times 10^{-3}$ & $0$ \\
\bottomrule
\end{tabular}
\caption{Computational results for $n=200$ and $r=3$.}
\label{tab:200r3}
\end{table}

\begin{table}[h!]
\centering
\footnotesize
\setlength{\tabcolsep}{4pt}
\begin{tabular}{l c c c c c}
\toprule
Method & Time  & RelErr  & $\lVert AX-\lambda X\rVert_F$ & $|\lambda-\lambda_{max}|$ & \#neg \\
\midrule
EIGS+SVD & $5.6$ & $4.8 \times 10^{-4}$ & $8.2 \times 10^{-3}$ & $3.8 \times 10^{-6}$ & $1464$ \\
EIGS+NMF & $5.6$ & $5.9 \times 10^{-4}$ & $1.4 \times 10^{-2}$ & $5.7 \times 10^{-6}$ & $0$ \\ 
EIGS & $5.6$ & $< 10^{-14}$ & $4.7 \times 10^{-12}$ & $< 10^{-14}$ & $0$ \\ \midrule
PSI & $3.3$ & $3.0 \times 10^{-2}$ & $7.8 \times 10^{-3}$ & $3.7 \times 10^{-6}$ & $984$ \\
RNeg & $3.7$ & $7.1 \times 10^{-1}$ & $1.0 \times 10^{-1}$ & $1.2 \times 10^{-3}$ & $0$ \\
\bottomrule
\end{tabular}
\caption{Computational results for $n=200$ and $r=4$.}
\label{tab:200r4}
\end{table}

We observe that the solutions of both PSI and EIGS+SVD have negative values. EIGS+NMF is in general more accurate compared to  RNeg.

%\begin{table}[h]\centering\small
%\begin{tabular}{lccccccc}\hline
%Method  & Time [s] & Min$(X)$ & \#neg & $\lambda$ & $|\lambda-\lambda_{max}|$ & Residual \\\hline
%EIGS  & 6.200 & -1.559e-02 & 47944 & 7.744e+00 & 0.000e+00 & 1.218e-05 \\
%PSI  & 62.267 & -5.071e-06 & 11737 & 7.744e+00 & 5.373e-04 & 7.076e-02 \\
%RNeg  & 16.895 & 5.519e-17 & 0 & 7.743e+00 & 1.189e-03 & 1.021e-01 \\
%%eig+SVD & 5.440e-04 & 6.237 & -1.558e-02 & 52288 & 7.744e+00 & 4.159e-06 & 7.962e-03 \\
%%eig+NMF & 1.387e+00 & 6.250 & 2.479e-08 & 0 & 7.707e+00 & 3.678e-02 & 2.074e+00 \\
%\hline\end{tabular}\label{tab:400r3}\caption{Computational results rank $3$ and $n = 400$.}
%\end{table}

%\begin{table}[h]\centering\small
%\begin{tabular}{lccccccc}\hline
%Method & $\|X-X^*\|_F$ & Time [s] & Min$(X)$ & \#neg & $\lambda$ & $|\lambda-\lambda_{max}|$ & Residual \\\hline
%EIGS & 0.000e+00 & 6.040 & -1.559e-02 & 47944 & 7.744e+00 & 0.000e+00 & 1.218e-05 \\
%PSI & 1.404e+00 & 32.931 & -6.991e-07 & 3880 & 7.744e+00 & 1.176e-05 & 7.743e-03 \\
%RNeg & 1.352e+00 & 18.031 & 5.580e-17 & 0 & 7.743e+00 & 1.189e-03 & 1.021e-01 \\
%%eig+SVD & 2.217e-04 & 6.101 & -1.559e-02 & 53378 & 7.744e+00 & 4.619e-07 & 2.199e-03 \\
%%eig+NMF & 1.392e+00 & 6.078 & 1.463e-07 & 0 & 5.892e+00 & 1.852e+00 & 8.536e+00 \\
%\hline\end{tabular}\caption{Computational results rank $4$ and $n = 400$.}
%\end{table}

\subsection{Experiment with a fully separable structure}

We introduce a separable spatial modulation of the discrete growth term, 
chosen specifically to favor low-rank structure in the solution, 
which allows efficient computation using low-rank factorization techniques. 
The semi-discrete system reads
\begin{equation}
\dot{Z} = \varepsilon \big( A Z + Z A^\top \big) 
+ r_0 \, Z + \varepsilon_r \, \phi \, Z \, \psi,
\end{equation}
where $A \in \mathbb{R}^{n \times n}$ is the finite-difference approximation 
of the one-dimensional Laplacian with Neumann boundary conditions, 
$Z \in \mathbb{R}^{n \times n}$ is the matrix of grid values $Z_{ij} = u(x_i,y_j)$, 
and 
\begin{equation}
\phi = \phi(x) \in \mathbb{R}^{n \times 1}, \qquad 
\psi = \psi(y) \in \mathbb{R}^{n \times 1}
\end{equation}
are column vectors defining the separable spatial modulation of the growth term in $x$ and $y$, respectively.  
In terms of low-rank factors $Z \approx U V^\top$, the operator is evaluated as
\begin{equation}
f(U,V) = \varepsilon \big( A U V^\top + U (A V)^\top \big) 
+ r_0 \, U V^\top + \varepsilon_r \, (\phi U) (V^\top \psi),
\end{equation}
where the separable growth term applies the modulation via standard matrix multiplication.
We consider  
\[
\phi(x) = 0.3 \, ( \sin(3\pi x)), \quad 
\psi(y) = 0.2 \, (\cos(\pi y)).
\]
and the following choices for the parameters
$   r_0 = 0.3,  \varepsilon = 0.1,$ and $ \varepsilon_r = 0.01$ and consider dimensions $n = 100, 150$ and ranks $r=2,3$, respectively.  The  choosen for the ODE-based methods is $h = 10^{-4}$.
The results are shown in Tables \ref{tab:m100r3} and \ref{tab:m150r2}.

\begin{table}[h!]
\centering
\footnotesize
\setlength{\tabcolsep}{4pt}
\begin{tabular}{l c c c c c}
\toprule
Method & Time & RelErr & $\lVert AX-\lambda X\rVert_F$ & $|\lambda-\lambda_{max}|$ & \#neg \\
\midrule
EIGS & $25.7$ & $< 10^{-14}$ & $2.65 \times 10^{-9}$ & $< 10^{-14}$ & $0$ \\ \midrule
EIGS+SVD & $25.8$ & $1.0 \times 10^{-11}$ & $1.1 \times 10^{-9}$ & $1.8 \times 10^{-13}$ & $0$ \\
EIGS+NMF & $25.8$ & $4.7 \times 10^{-5}$ & $1.4 \times 10^{-2}$ & $2.2 \times 10^{-7}$ & $0$ \\
PSI & $46.8$ & $8.1 \times 10^{-11}$ & $1.5 \times 10^{-10}$ & $3.0 \times 10^{-10}$ & $0$ \\
RNeg & $8.6$ & $1.7 \times 10^{-4}$ & $3.3 \times 10^{-4}$ & $3.8 \times 10^{-8}$ & $0$ \\
\bottomrule
\end{tabular}
\caption{Computational results for $n=100$ and $r=3$.}
\label{tab:m100r3}
\end{table}
\begin{table}[h!]
\centering
\footnotesize
\setlength{\tabcolsep}{4pt}
\begin{tabular}{l c c c c c}
\toprule
Method & Time & RelErr & $\lVert AX-\lambda X\rVert_F$ & $|\lambda-\lambda_{max}|$ & \#neg \\
\midrule
EIGS & $3967$ & $< 10^{-14}$ & $2.4 \times 10^{-9}$ & $< 10^{-14}$ & $0$ \\ \midrule
EIGS+SVD & $3967$ & $1.7 \times 10^{-10}$ & $6.9 \times 10^{-9}$ & $7.4 \times 10^{-13}$ & $0$ \\
EIGS+NMF & $3967$ & $4.7 \times 10^{-5}$ & $2.5 \times 10^{-2}$ & $1.1 \times 10^{-7}$ & $0$ \\
PSI & $162$ & $1.8 \times 10^{-10}$ & $6.9 \times 10^{-9}$ & $3.1 \times 10^{-10}$ & $0$ \\
RNeg & $29.6$ & $4.9 \times 10^{-5}$ & $2.9 \times 10^{-4}$ & $1.2 \times 10^{-8}$ & $0$ \\
\bottomrule
\end{tabular}
\caption{Computational results for $n = 150$ and $r=2$.}
\label{tab:m150r2}
\end{table}
In this case, the low-rank methods generally exhibit strong performance in terms of both accuracy and computational efficiency. PSI preserves nonnegativity, but this is not guaranteed as illustrated in the previous sections. 
RNeg is significantly faster than PSI, taking advantage of the low-rank structure directly and the stepsize adaptation strategy is advantageous when compared to PSI, which uses a fixed step. EIGS is very slow it cannot take advantage of the low-rank structure, and hence EIGS+SVD and EIGS+NMF are also significantly slower. 
Moreover, in this experiment, RNeg provides a better estimate to the eigenvalue than EIGS+NMF.

\section{Conclusion} 

In this paper, we proposed a nonnegative low-rank integrator, RNeg, see Algorithm~\ref{alg1}, to compute the leading eigenmatrix of a linear operation $\mathcal A$. We illustrated its use on two applications: Markov grids, a newly introduced class of Markov chains, and solving Meltzer operators arising from PDEs.  
%Further work include the use of RNeg in other applications, and the design of other algorithms to tackle this problem. 
The investigation of methods based on differential systems capable of handling both low-rank and nonnegativity constraints appears promising. Their design and use for other  problems is a topic for future research.

%{References for theoretical background:} %\cite{hundsdorfer2003numerical,farina2000positive,cantrell2003spatial,hanski1997metapopulation}.

\bibliographystyle{spmpsci}
\bibliography{biblio}

\end{document}